\documentclass[12pt, a4paper, oneside]{amsart}
\usepackage[english]{babel}
\usepackage{amsmath, amscd, amssymb, amsfonts, stmaryrd, vmargin, amsthm, mathrsfs}

\newtheorem{theorem}{Theorem}[section]
\newtheorem{lemma}[theorem]{Lemma}
\newtheorem{proposition}[theorem]{Proposition}
\newtheorem{corollary}[theorem]{Corollary}
\newtheorem{example}[theorem]{Example}
\newtheorem{remark}[theorem]{Remark}
\numberwithin{equation}{section}

\begin{document}
\title[On the structure of holomorphic isometric embeddings]{On the structure of holomorphic isometric embeddings of complex unit balls into bounded symmetric domains}
\author{Shan Tai Chan}
\address{Department of Mathematics, Syracuse University, Syracuse, NY 13244-1150, USA}
\curraddr{}
\email{schan08@syr.edu}
\thanks{}


\begin{abstract}
We study general properties of holomorphic isometric embeddings of complex unit balls $\mathbb B^n$ into bounded symmetric domains of rank $\ge 2$.
In the first part, we study holomorphic isometries from $(\mathbb B^n,kg_{\mathbb B^n})$ to $(\Omega,g_\Omega)$ with non-minimal isometric constants $k$ for any irreducible bounded symmetric domain $\Omega$ of rank $\ge 2$, where $g_D$ denotes the canonical K\"ahler-Einstein metric on any irreducible bounded symmetric domain $D$ normalized so that minimal disks of $D$ are of constant Gaussian curvature $-2$.
In particular, results concerning the upper bound of the dimension of isometrically embedded $\mathbb B^n$ in $\Omega$ and the structure of the images of such holomorphic isometries were obtained.

In the second part, we study holomorphic isometries from $(\mathbb B^n,g_{\mathbb B^n})$ to $(\Omega,g_\Omega)$ for any irreducible bounded symmetric domains $\Omega\Subset \mathbb C^N$ of rank equal to $2$ with $2N>N'+1$, where $N'$ is an integer such that $\iota:X_c\hookrightarrow \mathbb P^{N'}$ is the minimal embedding (i.e., the first canonical embedding) of the compact dual Hermitian symmetric space $X_c$ of $\Omega$.
We completely classify images of all holomorphic isometries from $(\mathbb B^n,g_{\mathbb B^n})$ to $(\Omega,g_\Omega)$ for $1\le n \le n_0(\Omega)$, where $n_0(\Omega):=2N-N'>1$.
In particular, for $1\le n \le n_0(\Omega)-1$ we prove that any holomorphic isometry from $(\mathbb B^n,g_{\mathbb B^n})$ to $(\Omega,g_\Omega)$ extends to some holomorphic isometry from $(\mathbb B^{n_0(\Omega)},g_{\mathbb B^{n_0(\Omega)}})$ to $(\Omega,g_\Omega)$.
\end{abstract}
\maketitle
\setcounter{page}{1}

\section{Introduction}\label{Sec:1}
In \cite{Ca53}, Calabi studied local holomorphic isometries from K\"ahler manifolds endowed with real-analytic metrics into complex space forms and their local rigidity.
Many results concerning local holomorphic isometric embeddings between bounded symmetric domains were obtained by Mok \cite{Mok2002, Mok11, Mok12, Mok16} and by Ng \cite{Ng10, Ng11}.
In \cite{CM16}, Mok and the author obtained a general result concerning general properties of the images of holomorphic isometric embeddings from $(\mathbb B^n,g_{\mathbb B^n})$ to $(\Omega,g_\Omega)$, where $g_D$ denotes the canonical K\"ahler-Einstein metric on $D$ normalized so that minimal disks of $D$ are of constant Gaussian curvature $-2$ for any irreducible bounded symmetric domain $D\Subset \mathbb C^N$ in its Harish-Chandra realization.
In addition, Mok and the author \cite{CM16} classified images of all holomorphic isometric embeddings from $(\mathbb B^m,g_{\mathbb B^m})$ to $(D^{\mathrm{IV}}_n,g_{D^{\mathrm{IV}}_n})$ for $1\le m\le n-1$ and $n\ge 3$, where $D^{\mathrm{IV}}_n$ denotes the type-$\mathrm{IV}$ domain (i.e., the Lie ball) of complex dimension $n$ (see Section \ref{Sec:2}).
On the other hand, Xiao-Yuan \cite{XY16} and Upmeier-Wang-Zhang \cite{UWZ16} classified all holomorphic isometric embeddings from $(\mathbb B^{n-1},g_{\mathbb B^{n-1}})$ to $(D^{\mathrm{IV}}_n,g_{D^{\mathrm{IV}}_n})$, $n\ge 3$, independently with explicit parametrizations.
Moreover, Xiao-Yuan \cite[Theorem 1.1.]{XY16} proved that any proper holomorphic map from the complex unit $m$-ball $\mathbb B^{m}$ to $D^{\mathrm{IV}}_n$, $n\ge 3$ and $m\le n-1$, with certain boundary regularities is a holomorphic isometric embedding provided that the codimension $n-m$ of the image of the $m$-ball sufficiently small and $m\ge 4$.

Throughout the present article, we also denote by $ds_U^2$ the Bergman metric of any bounded domain $U\Subset \mathbb C^N$ and we will simply use the term ``\emph{holomorphic isometries}" for holomorphic isometric embeddings.

Let $f:(\mathbb B^n,\lambda' g_{\mathbb B^n})\to (\Omega,g_\Omega)$ be a holomorphic isometry for some positive real constant $\lambda'$, where $\Omega$ is an irreducible bounded symmetric domain.
It is well-known that any bounded symmetric domain is equivalently a Hermitian symmetric space of the non-compact type and vice versa by the Harish-Chandra Embedding Theorem (cf. \cite{Wo72, Mok89}).
Then, it follows from \cite[Lemma 3]{CM16} that $\lambda'$ is a positive integer satisfying $1\le \lambda'\le r$, where $r:=\mathrm{rank}(\Omega)$ is the rank of $\Omega$ as a Hermitian symmetric space of the non-compact type.
Throughout the present article, we will call $\lambda'$ the \emph{isometric constant} of any given holomorphic isometry from $(\mathbb B^n,\lambda' g_{\mathbb B^n})$ to $(\Omega,g_\Omega)$.
In addition, given any holomorphic isometry $F:(\Delta,k ds_\Delta^2)\to(\Delta^p,ds_{\Delta^p}^2)$, we will call $k$ the \emph{isometric constant} of $F$, where $\Delta\Subset \mathbb C$ (resp. $\Delta^p\Subset \mathbb C^p$) denotes the open unit disk (resp. open unit polydisk) in the complex plane $\mathbb C$ (resp. the complex $p$-dimensional Euclidean space $\mathbb C^p$).

In the present article, we denote by $\widehat{\bf HI}_k(\mathbb B^n,\Omega)$ the space of all holomorphic isometries from $(\mathbb B^n,kg_{\mathbb B^n})$ to $(\Omega,g_\Omega)$, where $k$ is any positive integer satisfying $1\le k\le \mathrm{rank}(\Omega)$.
Motivated by \cite{Mok16, CM16}, we continue to study the structure of holomorphic isometries from $(\mathbb B^n,k g_{\mathbb B^n})$ to $(\Omega,g_\Omega)$ for any irreducible bounded symmetric domain $\Omega$ of rank $r\ge 2$ and any positive integer $k$ such that $1\le k\le r$.

In the first part, we consider the case where $k\ge 2$ is not the minimal isometric constant and obtain a similar result as \cite[Theorem 1]{CM16} when the isometric constant $k$ is equal to $2$. As a corollary of this result, we will also show that given any irreducible bounded symmetric domain $\Omega$ of rank at most three, all holomorphic isometries from $(\mathbb B^n,k g_{\mathbb B^n})$ to $(\Omega,g_\Omega)$ arise from linear sections of the minimal embedding of the compact dual Hermitian symmetric space $X_c$ of $\Omega$.

In the second part, the aim is to generalize our results in Chan-Mok \cite{CM16} for type-$\mathrm{IV}$ domains to more general irreducible bounded symmetric domains $\Omega$ of rank $2$. 
Let $\Omega\Subset \mathbb C^N$ be an irreducible bounded symmetric domain of rank $\ge 2$ in its Harish-Chandra realization.
In \cite{Mok16}, Mok proved that if 
$f:(\mathbb B^n,g_{\mathbb B^n})\to (\Omega,g_\Omega)$ is a holomorphic isometry, then $n\le p(\Omega)+1$, where $p(\Omega):=p(X_c)=p$ is defined by $c_1(X_c)=(p+2)\delta$ for the compact dual Hermitian symmetric space $X_c$ of $\Omega$ and the positive generator $\delta$ of $H^2(X_c,\mathbb Z)\cong \mathbb Z$ (cf. \cite{Mok16, CM16}).
By slicing of the complex unit ball $\mathbb B^{p(\Omega)+1}$ with affine linear subspaces $L$ of $\mathbb C^{p(\Omega)+1}$ such that $L\cap \mathbb B^{p(\Omega)+1}$ is non-empty,
we obtain many holomorphic isometries in $\widehat{\bf HI}_1(\mathbb B^n,\Omega)$ from any given holomorphic isometry
$F\in \widehat{\bf HI}_1(\mathbb B^{p(\Omega)+1},\Omega)$ for $n\le p(\Omega)$.
It is natural to ask whether all holomorphic isometries 
in $\widehat{\bf HI}_1(\mathbb B^n,\Omega)$
were obtained in that way for each $n\le p(\Omega)$.
In the case where $\Omega=D^{\mathrm{IV}}_N$ is the type-$\mathrm{IV}$ domain for some integer $N\ge 3$, the author and Mok (cf. \cite[Theorem 2]{CM16}) have shown that the answer is affirmative.
In general, this problem remains open.
In Chan-Mok \cite{CM16}, we showed that holomorphic isometries from $(\mathbb B^n,g_{\mathbb B^n})$ to $(\Omega,g_\Omega)$ arise from linear sections of the compact dual $X_c$ of $\Omega$, where $\Omega$ is an irreducible bounded symmetric domain of rank $\ge 2$. In general, we do not know whether this gives any relation between the spaces $\widehat{\bf HI}_1(\mathbb B^n,\Omega)$ and $\widehat{\bf HI}_1(\mathbb B^m,\Omega)$ for $1\le n< m \le p(\Omega)+1$ except the case where $\Omega=D^{\mathrm{IV}}_N$, $N\ge 3$, is the type-$\mathrm{IV}$ domain (cf. \cite{CM16}).
Recall that a type-$\mathrm{IV}$ domain is of rank $2$.
On the other hand, for a rank-$r$ irreducible bounded symmetric domain $\Omega$, any holomorphic isometry from $(\mathbb B^n,rg_{\mathbb B^n})$ to $(\Omega,g_\Omega)$ is totally geodesic
by the Ahlfors-Schwarz lemma (cf. \cite[Proposition 1]{CM16}).
In particular, we only need to consider the space $\widehat{\bf HI}_1(\mathbb B^n,\Omega)$ if $\Omega$ is of rank $2$.
Therefore, it is natural to study the problem when the target bounded symmetric domain $\Omega$ is of rank $2$.

In short, we will generalize the method in Chan-Mok \cite{CM16} for classifying images of all holomorphic isometries in $\widehat{\bf HI}_1(\mathbb B^n,D^{\mathrm{IV}}_N)$ for $N\ge 3$ and $n\ge 1$ to the study of images of holomorphic isometries in 
$\widehat{\bf HI}_1(\mathbb B^n,\Omega)$ for $1\le n\le n_0$ and certain irreducible bounded symmetric domains $\Omega\Subset\mathbb C^N$ of rank $2$ in their Harish-Chandra realizations, where $n_0=n_0(\Omega)>1$ is some integer depending on $\Omega$.
One of the key ingredients is the use of the explicit form of the polynomial $h_\Omega(z,z)$ as mentioned in \cite[Remark 1]{CM16}.
On the other hand, the author has found that the relation between $h_\Omega(z,\xi)$ and $\iota|_{\mathbb C^N}$ obtained from \cite{Lo77} has been written down explicitly by Fang-Huang-Xiao \cite{FHX16} for each irreducible bounded symmetric domain $\Omega$, where $\iota:X_c\hookrightarrow \mathbb P\big(\Gamma(X_c,\mathcal O(1))^*\big)\cong \mathbb P^{N'}$ is the minimal embedding.
Here $\mathcal O(1)$ is the positive generator of the Picard group $\mathrm{Pic}(X_c)\cong \mathbb Z$ of the compact dual $X_c$ of $\Omega$ and $\mathbb C^N\subset X_c$ is identified as a dense open subset of $X_c$ by the Harish-Chandra Embedding Theorem (cf. \cite{Mok89, Mok16, CM16}).
We refer the readers to \cite[Section 2.1]{CM16} for the background of bounded symmetric domains and their compact dual Hermitian symmetric spaces.

The main results in the first part of the present article are as follows.
\begin{proposition}\label{ProBoundofDimIsoConstGT2}
Let $\Omega\Subset \mathbb C^N$ be an irreducible bounded symmetric domain of rank $\ge 2$ in its Harish-Chandra realization and $\lambda'\ge 2$ be an integer.
If $\widehat{\bf HI}_{\lambda'}(\mathbb B^n,\Omega)\neq \varnothing$, then we have $n\le n_{\lambda'-1}(\Omega)$, where $n_{\lambda'-1}(\Omega)$ is the $(\lambda'-1)$-th null dimension of $\Omega$ (cf. \cite[p.\;253]{Mok89} and Section \ref{sec:2.1}).
\end{proposition}
\begin{theorem}\label{LinearSection_Ball_IsoConsteq2}
Let $\Omega\Subset \mathbb C^N$ be an irreducible bounded symmetric domain of rank $\ge 2$ in its Harish-Chandra realization.
We have the standard embeddings $\Omega\Subset \mathbb C^N\subset X_c$ of $\Omega$ as a bounded domain and its Borel embedding $\Omega\subset X_c$ as an open subset of its compact dual Hermitian symmetric space $X_c$ (see \cite[Theorem 1]{CM16}).
If $f\in \widehat{\bf HI}_2(\mathbb B^n,\Omega)$,
then $f(\mathbb B^n)$ is an irreducible component of $\mathscr V:=\mathscr V'\cap \Omega$ for some affine-algebraic subvariety $\mathscr V'\subset \mathbb C^N$ such that $\iota(\mathscr V)=P\cap \iota(\Omega)$,
where $P\subseteq \mathbb P\big(\Gamma(X_c,\mathcal O(1))^*\big)\cong \mathbb P^{N'}$ is some projective linear subspace and $\iota:X_c\hookrightarrow \mathbb P\big(\Gamma(X_c,\mathcal O(1))^*\big)\cong \mathbb P^{N'}$ is the minimal embedding.
\end{theorem}
\begin{theorem}\label{LS_rk_2or3}
Let $\Omega\Subset \mathbb C^N$ be an irreducible bounded symmetric domain of rank $r$ in its Harish-Chandra realization, where $2\le r\le 3$.
We have the standard embeddings $\Omega\Subset \mathbb C^N\subset X_c$ of $\Omega$ as a bounded domain and its Borel embedding $\Omega\subset X_c$ as an open subset of its compact dual Hermitian symmetric space $X_c$ (see \cite[Theorem 1]{CM16}).
If $f\in \widehat{\bf HI}_{\lambda'}(\mathbb B^n,\Omega)$, then $f(\mathbb B^n)$ is an irreducible component of some complex-analytic subvariety $\mathscr V \subset \Omega$ satisfying $\iota(\mathscr V)=P\cap \iota(\Omega)$, where $\iota:X_c\hookrightarrow \mathbb P\big(\Gamma(X_c,\mathcal O(1))^*\big)\cong\mathbb P^{N'}$ is the minimal embedding and $P\subseteq \mathbb P\big(\Gamma(X_c,\mathcal O(1))^*\big)\cong\mathbb P^{N'}$ is some projective linear subspace.
\end{theorem}

The main result of the second part is the following.
\begin{theorem}\label{Thm_Slicing_for_special_cases1}
Let $\Omega\Subset \mathbb C^N$ be an irreducible bounded symmetric domain of rank $2$ in its Harish-Chandra realization satisfying
$2N>N'+1$, where $N'$ $:=$ $\dim_{\mathbb C}$ $\mathbb P$$\big(\Gamma(X_c,\mathcal O(1))^*\big)$ and $X_c$ is the compact dual Hermitian symmetric space of $\Omega$.
Denote by $n_0(\Omega):=2N-N'$.
For $1\le n\le n_0(\Omega)-1$, if 
$f:(\mathbb B^n,g_{\mathbb B^n})\to (\Omega,g_\Omega)$ is a holomorphic isometric embedding,
then $f=F\circ \rho$ 
for some holomorphic isometric embeddings
$F:(\mathbb B^{n_0(\Omega)},g_{\mathbb B^{n_0(\Omega)}})\to (\Omega,g_\Omega)$ and 
$\rho:(\mathbb B^n,g_{\mathbb B^n})\to (\mathbb B^{n_0(\Omega)},g_{\mathbb B^{n_0(\Omega)}})$.
\end{theorem}
\begin{remark}\text{}
\begin{enumerate}
\item[(1)]
Theorem \ref{Thm_Slicing_for_special_cases1} actually asserts that any holomorphic isometric embedding $f\in \widehat{\bf HI}_1(\mathbb B^n,\Omega)$, $1\le n\le n_0(\Omega)-1$, extends to a holomorphic isometric embedding $F\in \widehat{\bf HI}_1(\mathbb B^{n_0(\Omega)},\Omega)$, where $\Omega\Subset \mathbb C^N$ is a rank-$2$ irreducible bounded symmetric domain satisfying $2N>N'+1$.
\item[(2)] We will show that for such irreducible bounded symmetric domains $\Omega$, $n_0(\Omega)$ $=$ $p(\Omega)+1$ only if $\Omega\cong D^{\mathrm{IV}}_N$ is the type-$\mathrm{IV}$ domain for some $N\ge 3$.
Therefore, one may regard this theorem as a generalization of Theorem 2 in Chan-Mok \cite{CM16} to holomorphic isometric embeddings from $(\mathbb B^n,g_{\mathbb B^n})$ to $(\Omega,g_\Omega)$ for any rank-$2$ irreducible bounded symmetric domain $\Omega$ satisfying $n_0(\Omega)>1$ and $1\le n\le n_0(\Omega)-1$.
\end{enumerate}
\end{remark}

\section{Preliminaries}\label{Sec:2}
Denote by $\lVert {\bf v} \rVert_{\mathbb C^n}$ the standard complex Euclidean norm of any vector ${\bf v}$ in $\mathbb C^n$.
The following lemma is a special case of a well-known result of Calabi \cite[Theorem 2 (Local Rigidity)]{Ca53}:
\begin{lemma}[cf. Calabi \cite{Ca53} or Lemma 3.3 in \cite{Ng11}]\label{LemCaSOS}
Let $g,f:B\to \mathbb C^N$ be holomorphic maps defined on some open subset $B\subset \mathbb C^n$ such that $\lVert f(w)\rVert^2_{\mathbb C^N}=\lVert g(w)\rVert^2_{\mathbb C^N}$ for any $w\in B$.
Then, there exists a unitary transformation $U$ in $\mathbb C^N$ such that $f=U\circ g$.
\end{lemma}
\begin{remark}
Writing $f=(f^1,\ldots,f^N)$ and $g=(g^1,\ldots,g^N)$, there exists an $N\times N$ unitary matrix ${\bf U'}$ such that
\[ {\bf U'} \begin{pmatrix}
f^1(w),\ldots, f^N(w)
\end{pmatrix}^T= \begin{pmatrix}
g^1(w),\ldots, g^N(w)
\end{pmatrix}^T \quad \forall\;w\in B. \]
\end{remark}
Moreover, we have the following fact from linear algebra.
\begin{lemma}[cf. Lemma 5 in \cite{CM16}]\label{Matrix2}
Let $m'$ and $n'$ be integers such that $1\le m'<n'$ and let ${\bf A''}\in M(m',n';\mathbb C)$ be such that ${\bf A''}\overline{{\bf A''}}^T={\bf I_{m'}}$.
Then, there exists ${\bf U'}\in M(n'-m',n';\mathbb C)$ such that $\begin{bmatrix}
{\bf U'}\\{\bf A''}
\end{bmatrix} \in U(n')$.
\end{lemma}

For the complex unit ball $\mathbb B^n\subset \mathbb C^n$, the K\"ahler form $\omega_{g_{\mathbb B^n}}$ of $(\mathbb B^n,g_{\mathbb B^n})$ is given by
\[ \omega_{g_{\mathbb B^n}}=-\sqrt{-1}\partial\overline\partial \log \left(1-\lVert w\rVert^2_{\mathbb C^n}\right) \]
so that $(\mathbb B^n,g_{\mathbb B^n})$ is of constant holomorphic sectional curvature $-2$.
Note that the Bergman metric $K_\Omega(z,\xi)$ of $\Omega$ can be expressed as $K_\Omega(z,\xi)={1\over \mathrm{Vol}(\Omega)} h_\Omega(z,\xi)^{-(p(\Omega)+2)}$, where $\mathrm{Vol}(\Omega)$ is the Euclidean volume of $\Omega\Subset \mathbb C^N$, $h_\Omega(z,\xi)$ is some polynomial in $(z,\overline\xi)$ such that $h_\Omega(z,{\bf 0})\equiv 1$ and $p(\Omega)$ is defined as in Section \ref{Sec:1}.
It follows from \cite{CM16} that the K\"ahler form $\omega_{g_{\Omega}}$ of $(\Omega,g_{\Omega})$ is given by
\[ \omega_{g_\Omega}=-\sqrt{-1}\partial\overline\partial\log h_\Omega(z,z) \]
in terms of the Harish-Chandra coordinates $z\in \Omega\Subset \mathbb C^N$.
The type-$\mathrm{IV}$ domain $D^{\mathrm{IV}}_N$, $N\ge 3$, is given by
\[ D^{\mathrm{IV}}_N = \left\{z=(z_1,\ldots,z_N)\in \mathbb C^N:
\sum_{j=1}^N |z_j|^2 < 2,\; \sum_{j=1}^N|z_j|^2 <1+ \left|{1\over 2}\sum_{j=1}^N z_j^2\right|^2 \right\} \]
(cf. \cite[p.\;83]{Mok89}).
Then, the K\"ahler form $\omega_{g_{D^{\mathrm{IV}}_N}}$ of $(D^{\mathrm{IV}}_N,g_{D^{\mathrm{IV}}_N})$ is given by
$\omega_{g_{D^{\mathrm{IV}}_N}}
= -\sqrt{-1}\partial\overline\partial \log \left(
1-\sum_{j=1}^N|z_j|^2 + \left|{1\over 2}\sum_{j=1}^N z_j^2\right|^2\right)$.
As mentioned in Section \ref{Sec:1}, we have the following:
For any irreducible bounded symmetric domain $\Omega\Subset \mathbb C^N$ of rank $r\ge 2$ in its Harish-Chandra realization, we may suppose that the Harish-Chandra coordinates $z=(z_1,\ldots,z_N)$ on $\Omega\Subset \mathbb C^N$ are chosen so that there are homogeneous polynomials $G_l(z)$ of $z$ and of degree $\deg(G_l)$, $1\le l \le N'$, such that 
\begin{enumerate}
\item[(i)] $2\le \deg (G_l) \le r$ for $N+1\le l\le N'$ and $G_j(z)=z_j$ for $1\le j\le N$,
\item[(ii)]
\[ h_\Omega(z,\xi) = 1+\sum_{j=1}^{N'}(-1)^{\deg (G_l)} G_l(z)\overline{G_l(\xi)} \]
and the restriction of the minimal embedding $\iota:X_c\hookrightarrow \mathbb P\big(\Gamma(X_c,\mathcal O(1))^*\big)\cong \mathbb P^{N'}$ to the dense open subset $\mathbb C^N \subset X_c$ may be written as
\[ \iota(z)=[1,G_1(z),\ldots,G_{N'}(z)] \]
in terms of the Harish-Chandra coordinates $z=(z_1,\ldots,z_N)\in \mathbb C^N$.
\end{enumerate}
For instance, if $\Omega=D^{\mathrm{IV}}_N\Subset \mathbb C^N$, $N\ge 3$, is the type-$\mathrm{IV}$ domain, then
$h_\Omega(z,z) = 1-\sum_{j=1}^N|z_j|^2 + \left|{1\over 2}\sum_{j=1}^N z_j^2\right|^2$
and $\iota(z) = \left[z_1,\ldots,z_N,1,{1\over 2}\sum_{j=1}^Nz_j^2\right]$ for $z=(z_1,\ldots,z_N)\in \mathbb C^N$ (cf. \cite[p.\;83]{Mok89}).
We refer the readers to Loos \cite{Lo77} and Fang-Huang-Xiao \cite{FHX16} for details of the above facts.

Let $f:(\mathbb B^n,kg_{\mathbb B^n})\to (\Omega,g_\Omega)$ be a holomorphic isometry such that $f({\bf 0})={\bf 0}$, where $\Omega$ is an irreducible bounded symmetric domain of rank $r\ge 2$ and $k$ is an integer such that $1\le k\le r$.
Then, we have the functional equation
\[ h_\Omega(f(w),f(w)) = \left(1-\lVert w\rVert_{\mathbb C^n}^2\right)^k \]
for $w\in \mathbb B^n$ (cf. \cite{Mok12, CM16}).

Throughout the present article, given any irreducible bounded symmetric domain $\Omega$, we always denote by $N':=\dim_{\mathbb C} \mathbb P\big(\Gamma(X_c,\mathcal O(1))^*\big)$, where $X_c$ is the compact dual Hermitian symmetric space of $\Omega$.

\subsection{On higher characteristic bundles over irreducible bounded symmetric domains}\label{sec:2.1}
Let $\Omega\Subset \mathbb C^N$ be an irreducible bounded symmetric domain of rank $r$ in its Harish-Chandra realization and $X_c$ be the compact dual of $\Omega$.
Throughout this section, we follow Wolf \cite{Wo72} and Mok \cite[pp.\;251-253]{Mok89}.
We always identify the base point $o\in X_0$ with ${\bf 0}\in \Omega=\xi^{-1}(X_0)$, where $\xi:\mathfrak m^+\cong \mathbb C^N\to G^{\mathbb C}/P\cong X_c$ is the embedding defined by $\xi(v)=\exp(v)\cdot P$ (cf. Wolf \cite{Wo72} and Mok \cite[p.\;94]{Mok89}).
Let $\Psi=\{\psi_1,\ldots,\psi_r\}\subset \Delta^+_M$ be a maximal strongly orthogonal set of non-compact positive roots (see \cite{Wo72}).
Then, we have the corresponding root vectors $e_{\psi_j}$, $1\le j\le r$.
Moreover, we have $\mathfrak g_{\psi_j}=\mathbb C e_{\psi_j}$ for $1\le j\le r$ and the maximal polydisk $\Delta^r\cong \Pi\subset \Omega$ is given by
$\Pi= \left(\bigoplus_{j=1}^r \mathfrak g_{\psi_j}\right)\cap \Omega$ (cf. \cite{Wo72, Mok14}).
From \cite[p.\;252]{Mok89}, for any $v\in \mathfrak m^+\cong T_{\bf 0}(\Omega)$, there exists $k\in \mathfrak k$ such that $\mathrm{ad}(k)\cdot v = \sum_{j=1}^r a_j e_{\psi_j}$ with $a_j\in \mathbb R$ ($1\le j\le r$) and $a_1\ge \cdots \ge a_r \ge 0$.
Then, $\eta=\sum_{j=1}^r a_j e_{\psi_j}$ is said to be the normal form of $v$ and is uniquely determined by $v$.
The cardinality of the set $\{j\in \{1,\ldots,r\}:a_j\neq 0\}$ is called the rank of $v$, which is denoted by $r(v)$.
For $1\le k\le r=\mathrm{rank}(\Omega)$, we define
\[ \mathcal S_{k,x}(\Omega):=\{[v]\in \mathbb P(T_x(\Omega)): 1\le r(v) \le k\} \subseteq \mathbb P(T_x(\Omega)), \]
called the $k$-th characteristic projective subvariety at $x\in \Omega$.
Then, $\mathcal S_k(\Omega):=\bigcup_{x\in \Omega} \mathcal S_{k,x}(\Omega) \subset \mathbb P T(\Omega)$ is called the $k$-th characteristic bundle over $\Omega$.
We simply call $\mathcal S_x(\Omega):=\mathcal S_{1,x}(\Omega)$ the characteristic variety at $x\in \Omega$.
From \cite{Mok89}, $\mathcal S_x(\Omega)\subset \mathbb P(T_x(\Omega))$ is a connected complex submanifold while $\mathcal S_{k,x}(\Omega)\subset \mathbb P(T_x(\Omega))$ is singular for $2\le k\le r-1$ provided that $r=\mathrm{rank}(\Omega)\ge 3$.
In addition, $\mathcal S_{r,x}(\Omega)=\mathbb P(T_x(\Omega))$ for $x\in \Omega$ and we have the inclusions $\mathcal S_{1,x}(\Omega)\subset \cdots \subset \mathcal S_{r,x}(\Omega)$.
Furthermore, for $r\ge 2$, $k\ge 2$ and $x\in \Omega$, $\mathcal S_{k,x}(\Omega)\subseteq \mathbb P(T_x(\Omega))$ is an irreducible projective subvariety because $\mathcal S_{k,x}(\Omega)\smallsetminus \mathcal S_{k-1,x}(\Omega)=P\cdot [v]$ is an orbit for any $[v]$ such that $v\in T_x(\Omega)\smallsetminus\{{\bf 0}\}$ is a rank-$k$ vector (cf. \cite{Mok2002Compositio}) and $\mathcal S_{k,x}(\Omega)\smallsetminus \mathcal S_{k-1,x}(\Omega)$ is dense in $\mathcal S_{k,x}(\Omega)$.
\begin{proposition}[cf. Mok \cite{Mok89}, p.\;252]
The $k$-th characteristic bundle $\mathcal S_k(\Omega)\to \Omega$ is holomorphic.
In addition, in terms of the Harish-Chandra embedding $\Omega\hookrightarrow \mathbb C^N$, $\mathcal S_k(\Omega)$ is parallel on $\Omega$ in the Euclidean sense, i.e., identifying $\mathbb PT(\Omega)$ with $\Omega\times \mathbb P^{N-1}$ using the Harish-Chandra coordinates, we have $\mathcal S_k(\Omega)\cong \Omega \times \mathcal S_{k,{\bf 0}}(\Omega)$.
\end{proposition}
\begin{remark}
For any nonzero vector $v\in T_{\bf 0}(\Omega)$, we let $\mathcal N_v:=\{\xi \in T_{\bf 0}(\Omega): R_{v\overline v \xi\overline \xi}(\Omega,g_\Omega) = 0\}$ be the null space of $v$.
From \cite{Mok89}, the $k$-th null dimension of $\Omega$ is defined by $n_k(\Omega):=\dim_{\mathbb C} \mathcal N_v = \dim_{\mathbb C} \mathcal N_\eta$, where $\eta=\sum_{j=1}^k a_j e_{\psi_j}$ ($a_j>0$ for $1\le j\le k$) is the normal form of some vector $v\in T_{\bf 0}(\Omega)$ of rank $k$.
Here $n_k(\Omega):=\dim_{\mathbb C} \mathcal N_v$ only depends on the rank $k=r(v)$ of $v$.
Then, Mok \cite{Mok89} proved that $\dim_{\mathbb C} \mathcal S_k(\Omega) = 2N-n_k(\Omega)-1$.
In particular, $\mathcal S_{k,x}(\Omega)$ is of dimension $N-n_k(\Omega)-1$ as an irreducible projective subvariety of $\mathbb P(T_x(\Omega))$ for any $x\in \Omega$.
Moreover, we have
$n(\Omega):=n_1(\Omega) \ge \cdots \ge n_r(\Omega)=0$ and $n(\Omega)$ is called the null dimension of $\Omega$.
From \cite{Mok89}, we define $p(\Omega)=\dim_{\mathbb C} \mathcal S_{\bf 0}(\Omega)$.
Then, we have $\dim_{\mathbb C}\Omega=N=p(\Omega)+n(\Omega)+1$.
\end{remark}
\noindent For $x\in \Omega$, under the identification
$T_x(\Omega)=T_x(X_c)$, we have $\mathcal S_x(\Omega) = \mathscr C_x(X_c)$, where $\mathscr C_y(X_c)\subset \mathbb P(T_y(X_c))$ is the variety of minimal rational tangents (VMRT) of the compact dual $X_c$ of $\Omega$ at $y\in X_c$.
We define $p(X_c):= \dim_{\mathbb C} \mathscr C_o(X_c)$ for the base point $o\in X_c$, which is identified with ${\bf 0}\in \mathfrak m^+$, i.e., $\xi({\bf 0})=o \in X_c\cong G^{\mathbb C}/P$.
For the notion of the VMRTs of Hermitian symmetric spaces of the compact type, we refer the reader to Hwang-Mok \cite{HM99}.
Note that $\dim_{\mathbb C} \mathscr C_y(X_c)$ does not depend on the choice of $y\in X_c$.
Then, we have $p(X_c)=p(\Omega)=\dim_{\mathbb C} \mathscr C_x(X_c)$ for any $x\in \Omega\subset X_c$.
\subsubsection{Holomorphic sectional curvature}\label{Sec:2.1.1}
Let $\Omega\Subset \mathbb C^N$ be an irreducible bounded symmetric domain of rank $r$ in its Harish-Chandra realization and $X_c$ be its compact dual Hermitian symmetric space.
Recall that $g_\Omega$ is the canonical K\"ahler-Einstein metric on $\Omega$ normalized so that minimal disks are of constant Gaussian curvature $-2$.
Then, the Bergman kernel on $\Omega$ is given by
\[ K_\Omega(z,\xi) = {1\over \mathrm{Vol}(\Omega)} h_\Omega(z,\xi)^{-(p(\Omega)+2)}, \]
where $\mathrm{Vol}(\Omega)$ is the Euclidean volume of $\Omega$ in $\mathbb C^N$, $h_\Omega(z,\xi)$ is a polynomial in $(z,\overline\xi)$ and $p(\Omega):=p(X_c)$ is the complex dimension of the VMRT of $X_c$ at the base point $o\in X_c$ (cf. \cite{Mok16}).
For $z\in \Omega\cong G_0/K$, there exists $k\in K$ such that $k\cdot z = \sum_{j=1}^r a_j e_{\psi_j} \in \left(\bigoplus_{j=1}^r \mathfrak g_{\psi_j}\right)\cap \Omega=\Pi$ for $|a_j|^2<1$, $1\le j\le r$, and
\[ h_\Omega(z,z) = \prod_{j=1}^r (1-|a_j|^2), \]
where $r$ is the rank of the irreducible bounded symmetric domain $\Omega$,
$\Pi\cong \Delta^r$ is a maximal polydisk in $\Omega$ which satisfies $(\Pi,g_\Omega|_\Pi)\cong (\Delta^r, {1\over 2} ds_{\Delta^r}^2)$ (cf. \cite{Mok14}).
In particular, it follows from the Polydisk Theorem (cf. \cite[p.\;88]{Mok89}) that
\[ -2 \le R_{\alpha\overline\alpha\alpha\overline\alpha}(\Omega,g_\Omega) \le -{2\over r} \]
for any unit vector $\alpha\in T_x(\Omega)$ and $x\in \Omega$.
Let $x\in \Omega$ and $\beta\in T_x(\Omega)$ be such that $\lVert \beta \rVert_{g_\Omega}^2=1$.
If $\beta$ is of rank $r(\beta) = s$, then we have
$R_{\beta\overline\beta\beta\overline\beta}(\Omega,g_\Omega)\le -{2\over s}$ because there exists $g\in G_0\cong \mathrm{Aut}_0(\Omega)$ such that $g\cdot \beta \in T_{\bf 0}(\Pi_s)$ for some totally geodesic submanifold $(\Pi_s,g_\Omega|_{\Pi_s})\subset (\Pi,g_\Omega|_\Pi)$ which is holomorphically isometric to $(\Delta^s,{1\over 2}ds_{\Delta^s}^2)$.

\section{On holomorphic isometries of complex unit balls into bounded symmetric domains with non-minimal isometric constants}
Let $\Omega\Subset \mathbb C^N$ be an irreducible bounded symmetric domain of rank $\ge 2$ in its Harish-Chandra realization.
In \cite{Mok16}, Mok studied the space $\widehat{\bf HI}_1(\mathbb B^n,\Omega)$ and provided a sharp upper bound on dimensions of isometrically embedded complex unit balls $(\mathbb B^n,g_{\mathbb B^n})$ in the irreducible bounded symmetric domain $(\Omega,g_\Omega)$ equipped with the canonical K\"ahler-Einstein metric $g_\Omega$.
Recall that given any $f \in \widehat{\bf HI}_k(\mathbb B^n,\Omega)$ with $k>0$ being a real constant, $k$ is a positive integer satisfying $1\le k\le \mathrm{rank}(\Omega)$ (cf. Chan-Mok \cite{CM16}).
It is natural to ask whether some results in Mok's study \cite{Mok16} could be generalized to the study of the space $\widehat{\bf HI}_k(\mathbb B^n,\Omega)$ for $k\ge 2$.

In the first part of this section (see Section \ref{Sec:3.1}), we provide an upper bound of $n$ whenever $\widehat{\bf HI}_k(\mathbb B^n,\Omega)\neq \varnothing$, where $k\ge 2$.
Note that such an upper bound is not sharp in general.
For instance, if $\Omega=D^{\mathrm{I}}_{p,q}$ with $q\ge p\ge 2$ and $k=\mathrm{rank}(\Omega)=p$, then $\widehat{\bf HI}_k(\mathbb B^n,\Omega)\neq \varnothing$ implies
$n \le {q\over p}$ (cf. \cite[Proposition 3.2]{KM08}).
On the other hand, our general result will imply that $n \le n_{p-1}(D^{\mathrm{I}}_{p,q}) = q-p+1$ whenever $\widehat{\bf HI}_p(\mathbb B^n,D^{\mathrm{I}}_{p,q})\neq \varnothing$ with $q\ge p\ge 2$.
In the case where $q=3$ and $p=2$, we have $n\le 2$ from our general result.
But then it follows from \cite[Proposition 3.2]{KM08} that $n = 1$ whenever $\widehat{\bf HI}_2(\mathbb B^n,D^{\mathrm{I}}_{2,3})\neq \varnothing$.
This explains that the upper bound obtained in our general result is not sharp in general.
However, one of the applications of our general result is that if $\Omega$ satisfies certain condition and $\widehat{\bf HI}_k(\mathbb B^n,\Omega)\neq \varnothing$ for some fixed real constant $k>1$, then $n\le p(\Omega)$.
In the second part of this section (see Section \ref{Sec:3.2}), we continue our study in Chan-Mok \cite{CM16} to the study of the space $\widehat{\bf HI}_2(\mathbb B^n,\Omega)$. In particular, we will obtain an analogue of \cite[Theorem 1]{CM16} for holomorphic isometries in the space $\widehat{\bf HI}_2(\mathbb B^n,\Omega)$ without using the system of functional equations introduced in Mok \cite{Mok12}.

\subsection{Upper bounds on dimensions of isometrically embedded complex unit balls in an irreducible bounded symmetric domain}\label{Sec:3.1}
Let $\Omega\Subset \mathbb C^N$ be an irreducible bounded symmetric domain of rank $\ge 2$.
Motivated by Mok's study \cite{Mok16}, one may continue to study the space $\widehat{\bf HI}_{\lambda'}(\mathbb B^n,\Omega)$ for $\lambda'>1$.
In this section, we study the upper bound on dimensions of isometrically embedded complex unit balls in an irreducible bounded symmetric domain of rank $\ge 2$ when the isometric constant is equal to $\lambda'>1$.
It is natural to ask whether the upper bound $p(\Omega)+1$ obtained in \cite{Mok16} is optimal in the sense that $n\le p(\Omega)+1$ whenever $\widehat{\bf HI}_{\lambda'}(\mathbb B^n,\Omega)\neq \varnothing$ for some real constant $\lambda'>0$. More specifically, we may ask whether $n\le p(\Omega)$ whenever $\widehat{\bf HI}_{\lambda'}(\mathbb B^n,\Omega)\neq \varnothing$ for some real constant $\lambda'>1$.

For any given integer $\lambda'\ge 2$, in order to obtain a sharp upper bound of $n$ such that $\widehat{\bf HI}_{\lambda'}(\mathbb B^n,\Omega)\neq \varnothing$, one should construct a holomorphic isometry $f\in \widehat{\bf HI}_{\lambda'}(\mathbb B^{n_0},\Omega)$ for some integer $n_0\ge 1$ such that $\widehat{\bf HI}_{\lambda'}(\mathbb B^{n},\Omega)$ $\neq$ $\varnothing$ only if $n\le n_0$.
Note that this problem remains unsolved, but we can provide a (rough) upper bound of $n$ by using the $k$-th characteristic bundle on $\Omega$.
More precisely, for any integer $\lambda'$ satisfying $2\le \lambda' \le \mathrm{rank}(\Omega)$, we prove that if $\widehat{\bf HI}_{\lambda'}(\mathbb B^n,\Omega)\neq \varnothing$, then $n\le n_{\lambda'-1}(\Omega)$, where $n_k(\Omega)$ is the $k$-th null dimension of $\Omega$ (cf. \cite{Mok89}).
This is precisely the assertion of Proposition \ref{ProBoundofDimIsoConstGT2}.
Moreover, for certain irreducible bounded symmetric domains $\Omega$ of rank $\ge 2$ (including the two irreducible bounded symmetric domains of the exceptional type) we will show that $n\le p(\Omega)$ whenever $\widehat{\bf HI}_{\lambda'}(\mathbb B^n,\Omega)\neq \varnothing$ for some integer $\lambda'\ge 2$.
Now, we are ready to prove Proposition \ref{ProBoundofDimIsoConstGT2}.
\begin{proof}[Proof of Proposition \ref{ProBoundofDimIsoConstGT2}]
Let $f\in \widehat{\bf HI}_{\lambda'}(\mathbb B^n,\Omega)$ be a holomorphic isometry.
Write $S:=f(\mathbb B^n)$.
If $\mathbb P(T_y(S))\cap \mathcal S_{\lambda'-1,y}(\Omega)\neq \varnothing$ for some $y\in S$, then there exists a vector $\alpha\in T_y(S)\subset T_y(\Omega)$ of unit length with respect to $g_\Omega$ and of rank $k \le \lambda'-1$ such that $R_{\alpha\overline\alpha\alpha\overline\alpha}(\Omega,g_\Omega)\le -{2\over k} \le -{2\over \lambda'-1}$ (cf. Section \ref{Sec:2.1.1}).
But then we have $-{2\over \lambda'}=R_{\alpha\overline\alpha\alpha\overline\alpha}(S,g_\Omega|_S) \le R_{\alpha\overline\alpha\alpha\overline\alpha}(\Omega,g_\Omega)\le -{2\over \lambda'-1}$ from the Gauss equation, which is a contradiction.
Hence, we have $\mathbb P(T_y(S))\cap \mathcal S_{\lambda'-1,y}(\Omega)=\varnothing$ for any $y\in S$.
Recall from Section \ref{sec:2.1} that $\mathcal S_{\lambda'-1,y}(\Omega)\subseteq \mathbb P(T_y(\Omega))$ is an irreducible projective subvariety of complex dimension $N-n_{\lambda'-1}(\Omega)-1$.
Then, it follows from the inequality
\[ \dim_{\mathbb C}(\mathbb P(T_y(S))\cap \mathcal S_{\lambda'-1,y}(\Omega))
\ge \dim_{\mathbb C}\mathbb P(T_y(S))+\dim_{\mathbb C} \mathcal S_{\lambda'-1,y}(\Omega) - \dim_{\mathbb C} \mathbb P(T_y(\Omega)) \]
that $n\le n_{\lambda'-1}(\Omega)$ (cf. \cite[p.\;57]{Mum95}).
\end{proof}
\begin{lemma}\label{NullDim1}
Let $\Omega\Subset \mathbb C^N$ be an irreducible bounded symmetric domain of rank $\ge 2$ in its Harish-Chandra realization.
Then, $n(\Omega)\le p(\Omega)$ if and only if $\Omega$ is biholomorphic to one of the following:
\begin{enumerate}
\item $D^{\mathrm{I}}_{p',q'}$, where $p'$ and $q'$ are integers satisfying $2=p' < q'$ or $p'=q'=3$,
\item $D^{\mathrm{II}}_m$ for some integer $m$ satisfying $5\le m\le 7$,
\item $D^{\mathrm{IV}}_n$ for some integer $n\ge 3$, 
\item $D^{\mathrm{V}}$,
\item $D^{\mathrm{VI}}$.
\end{enumerate}
\end{lemma}
\begin{proof}
From Mok \cite[pp.\;105-106]{Mok89}, we have $n(\Omega)+p(\Omega)+1=N$.
Then, the result follows from direct computations by the explicit data provided in Mok \cite[pp.\;249-251]{Mok89}.
\end{proof}
\begin{remark}
We observe that if $\Omega$ satisfies $n(\Omega)\le p(\Omega)$, then $\mathrm{rank}(\Omega)\le 3$.
In addition, Lemma \ref{NullDim1} implies that any irreducible bounded symmetric domain $\Omega$ of rank $2$ satisfies $n(\Omega)\le p(\Omega)$.
From \cite{Mok89}, it is clear that the condition $n(\Omega)\le p(\Omega)$ is equivalent to $\dim_{\mathbb C} \mathbb P(T_o(X_c)) \le 2 \cdot \dim_{\mathbb C} \mathscr C_o(X_c)$, where $X_c$ is the compact dual Hermitian symmetric space of $\Omega$ and $o\in X_c$ is a fixed base point.
\end{remark}
\noindent The following shows that for certain irreducible bounded symmetric domains $\Omega$ of rank $\ge 2$ and a fixed real constant $\lambda'>0$, $\widehat{\bf HI}_{\lambda'}(\mathbb B^{p(\Omega)+1},\Omega)\neq \varnothing$ if and only if $\lambda'=1$.
\begin{corollary}\label{CorNull_leq_VMRT}
Let $\Omega\Subset \mathbb C^N$ be an irreducible bounded symmetric domain of rank $\ge 2$ such that $n(\Omega) \le p(\Omega)$.
If $f\in \widehat{\bf HI}_{\lambda'}(\mathbb B^n,\Omega)$ for some real constant $\lambda'\ge 2$, then $n\le p(\Omega)$.
\end{corollary}
\begin{proof}
Note that $\lambda'$ is an integer satisfying $2\le \lambda' \le \mathrm{rank}(\Omega)$.
By the assumption, it follows from Proposition \ref{ProBoundofDimIsoConstGT2} that $n\le n_{\lambda'-1}(\Omega) \le n(\Omega)\le p(\Omega)$.
\end{proof}
\begin{remark}
Actually, Corollary \ref{CorNull_leq_VMRT} implies that the upper bound $p(\Omega)+1$ is optimal when the bounded symmetric domain $\Omega$ satisfies $n(\Omega)\le p(\Omega)$.
Moreover, the statement of Corollary \ref{CorNull_leq_VMRT} holds true for any irreducible bounded symmetric domain $\Omega$ of rank $2$.
\end{remark}
\subsubsection{Holomorphic isometries with the maximal isometric constant and applications}
Let $\Omega\Subset \mathbb C^N$ be an irreducible bounded symmetric domain of rank $r\ge 2$ in its Harish-Chandra realization.
Recall that if $f\in \widehat{\bf HI}_r(\mathbb B^n,\Omega)$, then $f$ is totally geodesic by the Ahlfors-Schwarz lemma.
The results obtained in Section \ref{Sec:3.1} can be applied so that we may prove that $n\le p(\Omega)$ without using the total geodesy of holomorphic isometries lying in the space $\widehat{\bf HI}_r(\mathbb B^n,\Omega)$.
\begin{proposition}\label{ProUBIsoConstr1}
Let $\Omega\Subset\mathbb C^N$ be an irreducible bounded symmetric domain of rank $r\ge 2$ in its Harish-Chandra realization such that $\Omega \not\cong D^{\mathrm{IV}}_3$ and let $f\in \widehat{\bf HI}_r(\mathbb B^n,\Omega)$.
Then, we have $n< p(\Omega)$.
If $F\in \widehat{\bf HI}_r(\mathbb B^n,\Omega)$, where $\Omega$ is an irreducible bounded symmetric domain of rank $r\ge 2$ and of tube type, then we have $n=1$.
\end{proposition}
\begin{proof}
Under the assumptions, Proposition \ref{ProBoundofDimIsoConstGT2} asserts that $n\le n_{r-1}(\Omega)$, so it remains to check that $n_{r-1}(\Omega)< p(\Omega)$ for any irreducible bounded symmetric domain $\Omega$ of rank $r\ge 2$ and $\Omega \not \cong D^{\mathrm{IV}}_3$.
Note that if $\Omega\cong D^{\mathrm{IV}}_3$, then $r=2$ and $n_{r-1}(\Omega)=1=p(\Omega)$.
It follows from \cite{Mok2002Compositio} that $\Omega$ is of tube type if and only if $n_{r-1}(\Omega)=1$ due to the dimension formula $\dim_{\mathbb C} \mathcal S_{r-1,x}(\Omega)$ $=$ $\dim_{\mathbb C} \mathbb P(T_x (\Omega)) - n_{r-1}(\Omega)$ of Mok \cite{Mok89}.
It is clear that if $\Omega$ is of tube type and $\Omega\not\cong D^{\mathrm{IV}}_3$, then $p(\Omega)>1$ so that $n_{r-1}(\Omega)=1<p(\Omega)$.
If $\Omega$ is of non-tube type, then $\Omega$ is biholomorphic to one of the following:
\begin{enumerate}
\item $D^{\mathrm{I}}_{p',q'}$ for some integers $p',q'$ satisfying $2\le p'<q'$,
\item $D^{\mathrm{II}}_{2m+1}$ for some integer $m\ge 2$,
\item $D^{\mathrm{V}}$.
\end{enumerate}
From the classification of the boundary components of bounded symmetric domains and the fact that $n_{r-1}(\Omega)$ is precisely the dimension of rank-$1$ boundary components of $\Omega$ (cf. \cite{Wo72} and \cite[p.\;298]{Mok2002Compositio}), we have
$n_{p'-1}(D^{\mathrm{I}}_{p',q'})=q'-p'+1< p(D^{\mathrm{I}}_{p',q'}) =p'+q'-2$ for $2\le p'<q'$,
$n_{m-1}(D^{\mathrm{II}}_{2m+1})=3 < p(D^{\mathrm{II}}_{2m+1}) = 2(2m-1)$ for $m\ge 2$ and
$n_1(D^{\mathrm V}) = 5 < p(D^{\mathrm{V}})=10$.
Hence, we have $n<p(\Omega)$.
On the other hand, given an irreducible bounded symmetric domain $\Omega$ of rank $r\ge 2$ and of tube type, if $F\in \widehat{\bf HI}_r(\mathbb B^n,\Omega)$, then we have $n\le n_{r-1}(\Omega)=1$, i.e., $n=1$.
\end{proof}
\noindent From the proof of Proposition \ref{ProUBIsoConstr1}, we have $n_{r-1}(\Omega)\le p(\Omega)$ for any irreducible bounded symmetric domain $\Omega$ of rank $r\ge 2$.
Given any irreducible bounded symmetric domain $\Omega$ of rank $r\ge 2$, we define
\[ \lambda_0(\Omega) := \min\{ \lambda\in \mathbb Z: 1\le \lambda\le r,\; n_{\lambda}(\Omega)\le p(\Omega)\}. \]
Then, we have $\lambda_0(\Omega)\le r-1$.
Note that $\Omega$ satisfies $n(\Omega)\le p(\Omega)$ if and only if $\lambda_0(\Omega)=1$.
Combining with Corollary \ref{CorNull_leq_VMRT}, we have the following:
\begin{theorem}\label{Thm_UB_PR}
Let $\Omega\Subset \mathbb C^N$ be an irreducible bounded symmetric domain of rank $r\ge 2$ in its Harish-Chandra realization and $\lambda'\ge 2$ be an integer.
If $\widehat{\bf HI}_{\lambda'}(\mathbb B^n,\Omega)\neq\varnothing$, then $n\le p(\Omega)$ provided that one of the following holds true:
\begin{enumerate}
\item $\Omega$ satisfies $n(\Omega)\le p(\Omega)$,
\item $\lambda'$ satisfies $\lambda_0(\Omega)+1\le \lambda' \le r$.
\end{enumerate}
\end{theorem}
\begin{proof}
If the bounded symmetric domain $\Omega$ satisfies $n(\Omega)\le p(\Omega)$, then the result follows from Corollary \ref{CorNull_leq_VMRT}.
If $\lambda'$ satisfies $\lambda_0(\Omega)+1\le \lambda' \le r$, then we have $n_{\lambda'-1}(\Omega)\le n_{\lambda_0(\Omega)}(\Omega) \le p(\Omega)$.
By Proposition \ref{ProBoundofDimIsoConstGT2}, we have $n\le n_{\lambda'-1}(\Omega) \le p(\Omega)$.
\end{proof}
\begin{remark}
If $\Omega$ satisfies $n(\Omega)\le p(\Omega)$, then $\lambda_0(\Omega)=1$ so that the condition $(2)$ does not provide an additional restriction on the given isometric constant $\lambda'$.
\end{remark}

\noindent In general, let $\Omega\Subset \mathbb C^N$ be an irreducible bounded symmetric domain of rank $\ge 2$ in its Harish-Chandra realization such that $n(\Omega)>p(\Omega)$.
Then, Lemma \ref{NullDim1} asserts that $\Omega$ is biholomorphic to one of the following:
\begin{enumerate}
\item $D^{\mathrm{I}}_{p,q}$ for some integers $p,q$ satisfying $3\le p\le q$ and $(p,q)\neq (3,3)$;
\item $D^{\mathrm{II}}_m$ for some integer $m\ge 8$;
\item $D^{\mathrm{III}}_m$ for some integer $m\ge 3$.
\end{enumerate}
In particular, we are able to compute $\lambda_0(\Omega)$ explicitly for each case.
\begin{table}[h]
\begin{center}
    \begin{tabular}{| c | c | c | c |}
    \hline
    Type & $\Omega$ & $\lambda_0(\Omega)$  \\ \hline
    ${\mathrm{I}}_{p,q}$ ($3\le p\le q$, $(p,q)\neq (3,3)$)&$D^{\mathrm{I}}_{p,q}$ & $\Big\lceil{(p+q)-\sqrt{(q-p)^2+4(p+q-2)}\over 2}\Big\rceil$\\ \hline 
    ${\mathrm{II}}_m$ ($m\ge 8$) & $D^{\mathrm{II}}_m$ & $
\lceil
\frac{(2m-1)-\sqrt{16m-31}}{4}
\rceil$\\ \hline     
    ${\mathrm{III}}_m$ ($m\ge 3$) & $D^{\mathrm{III}}_{m}$ & $
\lceil
{(2m+1)-\sqrt{8m-7}\over 2}
\rceil$ \\ \hline        
    \end{tabular}
     \caption{The formula of $\lambda_0(\Omega)$}
     \label{tab:table1}
\end{center}
\end{table}

\noindent
Here $\lceil x \rceil$ denotes the smallest integer greater than or equal to $x$ for any real number $x$.
\begin{example}
If $\Omega=D^{\mathrm{III}}_{7}$, then $\Omega$ is of rank $7$,
$n_k(\Omega) = {1\over 2}(7-k)(7-k+1)$ and $p(\Omega)=6$ (cf. \cite[p.\;86;\;p.\;250]{Mok89})
so that $\lambda_0(\Omega)=4 =\mathrm{rank}(\Omega)-3$.
Given any integer $\lambda'$ satisfying $5\le \lambda' \le 7$, Theorem \ref{Thm_UB_PR} asserts that $n\le p(\Omega)=6$ whenever $\widehat{\bf HI}_{\lambda'}(\mathbb B^n,D^{\mathrm{III}}_{7})\neq \varnothing$.

In general, by using the expression of $\lambda_0(D^{\mathrm{III}}_{m+2})$ in terms of $m$ for any integer $m\ge 1$ (see Table \ref{tab:table1}), one observes that the sequence $\left\{\mathrm{rank}(D^{\mathrm{III}}_{m+2}) - (\lambda_0(D^{\mathrm{III}}_{m+2})+1)\right\}_{m=1}^{+\infty}$ is monotonic increasing and 
$a_m:=\mathrm{rank}(D^{\mathrm{III}}_{m+2}) - (\lambda_0(D^{\mathrm{III}}_{m+2})+1) \to +\infty$
as $m\to +\infty$.
Moreover, ${a_m\over \mathrm{rank}(D^{\mathrm{III}}_{m+2})} \to 0$ as $m\to +\infty$.
That means $\mathrm{rank}(D^{\mathrm{III}}_{m+2})$ grows much faster than $a_m$ as $m$ is increasing.
This shows that in general the range of the isometric constants $\lambda'$ mentioned in condition $(2)$ of Theorem \ref{Thm_UB_PR} is quite restrictive for a rank-$r$ irreducible bounded symmetric domain $\Omega$, $r\ge 2$, such that $n(\Omega)>p(\Omega)$.
\end{example}
\subsection{Holomorphic isometries with the isometric constant equal to two and applications}\label{Sec:3.2}
Let $\Omega\Subset \mathbb C^N$ be an irreducible bounded symmetric domain of rank $\ge 2$ in its Harish-Chandra realization and $X_c$ be the compact dual Hermitian symmetric space of $\Omega$.
Then, it follows from the observation in Section \ref{Sec:2} that the polynomial $h_\Omega(z,z)$ can be written as
\[ h_\Omega(z,z)
=1-\sum_{l=1}^{m_1(\Omega)} |G_l^{(1)}(z)|^2+\sum_{l'=1}^{m_2(\Omega)}|G_{l'}^{(2)}(z)|^2, \]
where $G_l^{(1)}(z),G_{l'}^{(2)}(z)$ are homogeneous polynomials of $z$ and $m_1(\Omega),m_2(\Omega)$ are positive integers depending on $\Omega$ such that
\begin{enumerate}
\item $m_1(\Omega)+m_2(\Omega) = N'$ and $m_1(\Omega)\ge N$,
\item $\deg (G_l^{(1)})$ ($1\le l\le m_1(\Omega)$) is odd while $\deg (G_{l'}^{(2)})\ge 2$ ($1\le l'\le m_2(\Omega)$) is even,
\item $G_j^{(1)}(z)=z_j$ for $1\le j\le N$.
\end{enumerate}
Moreover, in terms of the Harish-Chandra coordinates $z=(z_1,\ldots,z_N)\in \mathbb C^N$, the restriction of $\iota$ to the dense open subset $\mathbb C^N\subset X_c$ may be written as
\[ \iota(z_1,\ldots,z_N)
=\left[1,G_{1}^{(1)}(z),\ldots,G_{m_1(\Omega)}^{(1)}(z),
G_{1}^{(2)}(z),\ldots,G_{m_2(\Omega)}^{(2)}(z)\right] \]
up to reparametrizations, where $\iota:X_c\hookrightarrow \mathbb P\big(\Gamma(X_c,\mathcal O(1))^*\big)\cong \mathbb P^{N'}$ is the minimal embedding.
\begin{remark}
As mentioned in Section \ref{Sec:2}, the above observation can be obtained from \cite{Lo77} and has been written down explicitly by Fang-Huang-Xiao \cite{FHX16}.
\end{remark}

In Chan-Mok \cite{CM16}, we studied images of holomorphic isometries in $\widehat{\bf HI}_{\lambda'}(\mathbb B^n,\Omega)$ when $\lambda'=1$.
However, it is not obvious how the method in Chan-Mok \cite{CM16} could be generalized to the study of images of holomorphic isometries in $\widehat{\bf HI}_{\lambda'}(\mathbb B^n,\Omega)$ for $\lambda'>1$ so as to obtain an analogue of Theorem 1 in \cite{CM16} for all holomorphic isometries in $\widehat{\bf HI}_{\lambda'}(\mathbb B^n,\Omega)$ and for any $\lambda'>0$.
After that, we observe that the above explicit form of $h_\Omega(z,z)$ is useful for continuing the study of images of holomorphic isometries in $\widehat{\bf HI}_{\lambda'}(\mathbb B^n,\Omega)$ when the isometric constant $\lambda'$ equals $2$.
Recall that Theorem \ref{LinearSection_Ball_IsoConsteq2} is exactly an analogue of Theorem 1 in \cite{CM16} for all holomorphic isometries in $\widehat{\bf HI}_{2}(\mathbb B^n,\Omega)$.
We are now ready to prove Theorem \ref{LinearSection_Ball_IsoConsteq2}.
\begin{proof}[Proof of Theorem \ref{LinearSection_Ball_IsoConsteq2}]
Let $f:(\mathbb B^n,2g_{\mathbb B^n})\to (\Omega,g_\Omega)$ be a holomorphic isometric embedding, where $\Omega\Subset \mathbb C^N$ is an irreducible bounded symmetric domain of rank $\ge 2$ in its Harish-Chandra realization.
Assume without loss of generality that $f({\bf 0})={\bf 0}$.
Then, we have the functional equation
\begin{equation}\label{EqCh8FE1}
\begin{split}
&1-\sum_{l=1}^{m_1(\Omega)} |G_l^{(1)}(f(w))|^2+\sum_{l=1}^{m_2(\Omega)}|G_l^{(2)}(f(w))|^2\\
=& \left(1-\sum_{\mu=1}^n |w_\mu|^2 \right)^2
=1- \sum_{\mu=1}^n |\sqrt{2}w_\mu|^2 + \sum_{1 \le \mu,\mu' \le n} |w_{\mu} w_{\mu'}|^2
\end{split}
\end{equation}
for $w\in \mathbb B^n$ and the polarized functional equation
\begin{equation}\label{EqCh8PFE1}
1-\sum_{l=1}^{m_1(\Omega)} G_l^{(1)}(f(w))\overline{G_l^{(1)}(f(\zeta))}+\sum_{l=1}^{m_2(\Omega)}G_l^{(2)}(f(w))\overline{G_l^{(2)}(f(\zeta))}
= \left(1-\sum_{\mu=1}^n w_\mu\overline{\zeta_\mu} \right)^2
\end{equation}
for $w,\zeta\in \mathbb B^n$.
We write
$\sum_{1 \le \mu,\mu' \le n} |w_{\mu} w_{\mu'}|^2 = \sum_{l=1}^{m_0} |\Xi_l(w)|^2$ for some homogeneous polynomials $\Xi(w)$ of degree $2$ and $m_0:= {n(n+1)\over 2}$.
Moreover, we write ${\bf G}^{(j)}(z)= \begin{pmatrix}
G^{(j)}_1(z),\ldots, G^{(j)}_{m_j(\Omega)}(z)
\end{pmatrix}^T$ for $j=1,2$.
Let $N_0:=\max\{n+m_2(\Omega),m_0+m_1(\Omega)\}$.
Then, there exists ${\bf U}\in U(N_0)$ such that
\begin{equation}\label{eq_var_iso_relate1}
{\bf U} \cdot \begin{pmatrix}
\sqrt{2} w_1 \\ \vdots \\ \sqrt{2} w_n \\
{\bf G}^{(2)}(f(w)) \\ {\bf 0}_{(N_0-n-m_2(\Omega))\times 1}
\end{pmatrix}
= \begin{pmatrix}
\Xi_1(w) \\ \vdots \\ \Xi_{m_0}(w) \\
{\bf G}^{(1)}(f(w))\\{\bf 0}_{(N_0 - m_1(\Omega)-m_0)\times 1}
\end{pmatrix}
\end{equation}
by Lemma \ref{LemCaSOS} and Eq. (\ref{EqCh8FE1}).
We write ${\bf U}=\begin{bmatrix} {\bf U_1} \\ {\bf U_2} \end{bmatrix}$ with ${\bf U_1}\in M(m_0,N_0;\mathbb C)$ and ${\bf U_2}\in M(N_0-m_0,N_0;\mathbb C)$.
We also write ${\bf U_2} = \begin{bmatrix}
{\bf U_{21}} & {\bf U_{22}}
\end{bmatrix}$ with ${\bf U_{21}}\in M(N_0-m_0,n;\mathbb C)$ and ${\bf U_{22}}\in M(N_0-m_0,N_0-n;\mathbb C)$.
Denote by $(Jf)(w)$ the complex Jacobian matrix of the holomorphic map $f:\mathbb B^n \to \Omega\Subset \mathbb C^N$ at $w\in \mathbb B^n$.
Note that we have
\begin{equation}\label{eq_var_iso_relate2}
\overline{(Jf)({\bf 0})}^T
\begin{pmatrix}
f^1(w),\ldots, f^N(w)
\end{pmatrix}^T = 2 \begin{pmatrix}
w_1,\ldots, w_n
\end{pmatrix}^T
\end{equation}
by differentiating
$\text{Eq.}\;(\ref{EqCh8PFE1})$ with respect to $\overline \zeta_\mu$ at $\zeta={\bf 0}$ for $1\le \mu\le n$.
In addition, $(Jf)({\bf 0})
 \in M(N,n;\mathbb C)$ is of rank $n$.
Moreover, we have $\sqrt{2} {\bf U_{21}} = \begin{pmatrix}
(Jf)({\bf 0}) \\ {\bf 0}_{(N_0-m_0-N)\times n}
\end{pmatrix}$ and $\overline{(Jf)({\bf 0})}^T (Jf)({\bf 0})=2 {\bf I_n}$.
Therefore, it follows from
Eq. (\ref{eq_var_iso_relate1})
and
Eq. (\ref{eq_var_iso_relate2})
that
\begin{equation}\label{EqCh8Image1}
\begin{bmatrix}
\begin{bmatrix}
{1\over 2}(Jf)({\bf 0})\overline{(Jf)({\bf 0})}^T\\ {\bf 0}_{(N_0-m_0-N)\times N}
\end{bmatrix}
& {\bf U_{22}}
\end{bmatrix} \begin{pmatrix}
{\bf f}(w)\\ {\bf G}^{(2)}(f(w))\\ {\bf 0}_{(N_0-n-m_2(\Omega))\times 1}
\end{pmatrix}
= \begin{pmatrix}
{\bf G}^{(1)}(f(w))\\ {\bf 0}_{(N_0-m_0-m_1(\Omega))\times 1}
\end{pmatrix}
\end{equation}
for any $w\in \mathbb B^n$, where ${\bf f}(w):=\begin{pmatrix}
f^1(w),\ldots, f^N(w)
\end{pmatrix}^T$.
Writing ${\bf B}:=\begin{bmatrix}
{\bf \hat U_{21}}
& {\bf U_{22}}
\end{bmatrix}$ with ${\bf \hat U_{21}}=\begin{bmatrix}
{1\over 2}(Jf)({\bf 0})\overline{(Jf)({\bf 0})}^T\\ {\bf 0}_{(N_0-m_0-N)\times N}
\end{bmatrix}$, we define
\begin{equation}\label{Eq:AffineAlgeVar}
\mathscr V':=
\left\{z\in \mathbb C^N:
{\bf B}\begin{pmatrix}
z_1\\\vdots \\ z_N \\ {\bf G}^{(2)}(z)\\ {\bf 0}_{(N_0-n-m_2(\Omega))\times 1}
\end{pmatrix}
= \begin{pmatrix}
{\bf G}^{(1)}(z)\\ {\bf 0}_{(N_0-m_0-m_1(\Omega))\times 1}
\end{pmatrix}\right\}
\end{equation}
and $\mathscr V:=\mathscr V'\cap \Omega$.
Then, we have $f(\mathbb B^n)\subseteq \mathscr V$ by $\text{Eq.}\; (\ref{EqCh8Image1})$.
Note that the tangential dimension $\mathrm{tdim}_{\bf 0}\mathscr V$ of $\mathscr V$ at ${\bf 0}$ is less than or equal to
$N$ $-$ $\mathrm{rank}\left({1\over 2}(Jf)({\bf 0})\overline{(Jf)({\bf 0})}^T - {\bf I_N}\right)$.
From \cite[p.\;49]{FuzhenZhang}, we have
\[ \mathrm{rank}\left({1\over 2}(Jf)({\bf 0})\overline{(Jf)({\bf 0})}^T - {\bf I_N}\right) \ge \left|\mathrm{rank} \left({1\over 2}(Jf)({\bf 0})\overline{(Jf)({\bf 0})}^T\right) - \mathrm{rank} \;{\bf I_N}\right| = N-n. \]
On the other hand, $\left({1\over 2}(Jf)({\bf 0})\overline{(Jf)({\bf 0})}^T - {\bf I_N}\right)\cdot (Jf)({\bf 0}) = {\bf 0}$ so that
\[ 0 \ge \mathrm{rank} \left({1\over 2}(Jf)({\bf 0})\overline{(Jf)({\bf 0})}^T - {\bf I_N}\right) + \mathrm{rank} (Jf)({\bf 0}) - N \]
and thus $\mathrm{rank} \left({1\over 2}(Jf)({\bf 0})\overline{(Jf)({\bf 0})}^T - {\bf I_N}\right) \le N-n$.
Therefore, we have
\[ \mathrm{rank} \left({1\over 2}(Jf)({\bf 0})\overline{(Jf)({\bf 0})}^T - {\bf I_N}\right)=N-n.\]
Moreover, $\mathscr V$ contains $f(\mathbb B^n)$ and ${\bf 0}\in f(\mathbb B^n)$, thus $\dim_{\bf 0}\mathscr V \ge n \ge \mathrm{tdim}_{\bf 0} \mathscr V$.
Note that $\dim_{\bf 0} \mathscr V \le \mathrm{tdim}_{\bf 0} \mathscr V$.
Hence, we have $\dim_{\bf 0}\mathscr V = \mathrm{tdim}_{\bf 0} \mathscr V=n$ and thus $\mathscr V$ is smooth at ${\bf 0}$.
Let $S$ be the irreducible component of $\mathscr V$ containing $f(\mathbb B^n)$.
Then, we have $\dim S= n = \dim f(\mathbb B^n)$ and thus $S=f(\mathbb B^n)$ because both $S$ and $f(\mathbb B^n)$ are irreducible complex-analytic subvarieties of $\mathscr V$ containing the smooth point ${\bf 0}\in \mathscr V$ of $\mathscr V$.
In particular, $f(\mathbb B^n)$ is the irreducible component of $\mathscr V$ containing ${\bf 0}$.
Moreover, it is clear that $\mathscr V'\subset \mathbb C^N$ is an affine-algebraic subvariety and
$\iota(\mathscr V)=P\cap \iota(\Omega)$,
where
\begin{equation}\label{Eq.ProjSubspace}
P:=
\left\{[\xi_0,\xi_1,\ldots,\xi_{N'}]\in \mathbb P^{N'}:
{\bf B} \begin{pmatrix}
\xi_1\\\vdots \\ \xi_N \\ \xi_{m_1(\Omega)+1} \\ \vdots \\ \xi_{N'} \\ {\bf 0}_{(N_0-n-m_2(\Omega))\times 1}
\end{pmatrix}
= \begin{pmatrix}
\xi_1 \\ \vdots \\  \xi_{m_1(\Omega)} \\ {\bf 0}_{(N_0-m_0-m_1(\Omega))\times 1}
\end{pmatrix}\right\}
\end{equation}
is a projective linear subspace of $\mathbb P^{N'}$.
\end{proof}

\subsubsection{On holomorphic isometries from the Poincar\'e disk into polydisks}
The author \cite{Ch16} and Ng \cite{Ng10} studied the classification problem of all holomorphic isometries from the Poincar\'e disk into the $p$-disk with any isometric constant $k$, $1\le k\le p$, and $p\ge 2$. The classification problem remains unsolved when $p\ge 5$.
In this section, we consider the structure of images of such holomorphic isometries for $k\le 2$ and obtain an analogue of Theorem \ref{LinearSection_Ball_IsoConsteq2} when the domain is the Poincar\'e disk and the target is the $p$-disk for some $p\ge 2$.

Note that the restriction $\varrho$ of the Segre embedding
$\varsigma:(\mathbb P^1)^p \hookrightarrow \mathbb P^{2^{p}-1}$ to the dense open subset $\mathbb C^p\subset (\mathbb P^1)^p$ is given by
\[ \varrho(z_1,\ldots,z_p)
= \varsigma([1,z_1],\ldots,[1,z_p]) \]
in terms of the standard holomorphic coordinates $z=(z_1,\ldots,z_p)\in \mathbb C^p$.
Here $\mathbb C^p$ is identified with its image $\xi(\mathbb C^p)$ in $(\mathbb P^1)^p$, where the map $\xi: \mathbb C^p \hookrightarrow (\mathbb P^1)^p$ is defined by $\xi(z_1,\ldots,z_p):= ([1,z_1],\ldots,[1,z_p])$.

Actually, the author \cite{Ch16} observed that the following can be proved by the same method as the proof of Theorem 1 in \cite{CM16}.
\begin{proposition}[cf. Proposition 5.2.4. in \cite{Ch16}] \label{Pro:HI_Disk_p-disk_IC1}
Let $f:(\Delta,ds_\Delta^2)\to (\Delta^p,ds_{\Delta^p}^2)$ be a holomorphic isometric embedding, where $p\ge 2$ is an integer.
Then, $f(\Delta)$ is an irreducible component of $\mathscr V \cap \Delta^p$ for some affine-algebraic subvariety $\mathscr V\subset \mathbb C^p$ such that $\varrho(\mathscr V\cap \Delta^p) = \varrho(\Delta^p)\cap P$, where $P\subseteq \mathbb P^{2^p-1}$ is a projective linear subspace.
\end{proposition}

Similarly, we observe that the method in the proof of Theorem \ref{LinearSection_Ball_IsoConsteq2} is also valid for any holomorphic isometry from $(\Delta,2ds_\Delta^2)$ to $(\Delta^p,ds_{\Delta^p}^2)$, where $p\ge 2$.
\begin{proposition}\label{Pro:HI_Disk_p-disk_IC2}
Let $f:(\Delta,2ds_\Delta^2)\to (\Delta^p,ds_{\Delta^p}^2)$ be a holomorphic isometric embedding, where $p\ge 2$ is an integer.
Then, $f(\Delta)$ is an irreducible component of $\mathscr V \cap \Delta^p$ for some affine-algebraic subvariety $\mathscr V\subset \mathbb C^p$ such that $\varrho(\mathscr V\cap \Delta^p) = \varrho(\Delta^p)\cap P$, where $P\subseteq \mathbb P^{2^p-1}$ is a projective linear subspace.
\end{proposition}
\begin{proof}
Assume without loss of generality that $f({\bf 0})={\bf 0}$.
Note that
\begin{equation}\label{Eq3.6}
\begin{split}
&h_{\Delta^p}(z,z) = \prod_{j=1}^p (1-|z_j|^2)\\
=&1 - \sum_{n=1}^{\lfloor{p+1\over 2}\rfloor} \sum_{1\le i_1 <\cdots < i_{2n-1}\le p} |z_{i_1}\cdots z_{i_{2n-1}}|^2
+ \sum_{n=1}^{\lfloor {p\over 2} \rfloor} \sum_{1 \le j_1 <\cdots < j_{2n}\le p} |z_{j_1}\cdots z_{j_{2n}}|^2.
\end{split}
\end{equation}
In the proof of Theorem \ref{LinearSection_Ball_IsoConsteq2}, we put $n=1$ and replace the term $\sum_{l=1}^{m_1(\Omega)} |G_l^{(1)}(z)|^2$ (resp. $\sum_{l=1}^{m_2(\Omega)} |G_l^{(2)}(z)|^2$) by
\[ \sum_{n=1}^{\lfloor{p+1\over 2}\rfloor} \sum_{1\le i_1 <\cdots < i_{2n-1}\le p} \left|\prod_{\mu=1}^{2n-1}z_{i_\mu}\right|^2\qquad \left(\text{resp.}\;\sum_{n=1}^{\lfloor {p\over 2} \rfloor} \sum_{1 \le j_1 <\cdots < j_{2n}\le p} \left|\prod_{\mu=1}^{2n} z_{j_\mu}\right|^2 \;\right).\]
Indeed, we may define $m_1(\Delta^p)$ and $m_2(\Delta^p)$.
Then, we compute $m_1(\Delta^p)=m_2(\Delta^p)+1=2^{p-1}$.
In this situation, the integer $N_0$ defined in the proof of Theorem \ref{LinearSection_Ball_IsoConsteq2} is equal to $m_1(\Delta^p)+1 =2^{p-1}+1$.
Then, the result follows directly from the arguments in the proof of Theorem \ref{LinearSection_Ball_IsoConsteq2}.
\end{proof}

\subsubsection{On holomorphic isometries of complex unit balls into irreducible bounded symmetric domains of rank at most $3$}
Given an irreducible bounded symmetric domain $\Omega\Subset \mathbb C^N$ of rank $\ge 2$, it is natural to ask whether all holomorphic isometries in $\widehat{\bf HI}(\mathbb B^n,\Omega)$ arise from linear sections of the minimal embedding of the compact dual $X_c$ of $\Omega$ in general.
In Chan-Mok \cite{CM16}, we showed that the answer is affirmative for all holomorphic isometries in $\widehat{\bf HI}_{\lambda'}(\mathbb B^n,\Omega)$ whenever $\widehat{\bf HI}_{\lambda'}(\mathbb B^n,\Omega)\neq \varnothing$ and $\lambda'\in \{1,\mathrm{rank}(\Omega)\}$.
On the other hand, Theorem \ref{LinearSection_Ball_IsoConsteq2} asserts that the answer is also affirmative for all holomorphic isometries in $\widehat{\bf HI}_{2}(\mathbb B^n,\Omega)$ whenever $\widehat{\bf HI}_{2}(\mathbb B^n,\Omega)\neq \varnothing$.
In other words, we may prove Theorem \ref{LS_rk_2or3} as follows.
\begin{proof}[Proof of Theorem \ref{LS_rk_2or3}]
Recall that $\lambda'$ is an integer satisfying $1\le \lambda' \le r$ (cf. \cite[Lemma 3]{CM16})
If $r=2$, then $\lambda'=1$ or $\lambda'=2$.
In the case of $\lambda'=1$, the result follows from \cite[Theorem 1]{CM16}.
When $\lambda'=2$, we may suppose that $f({\bf 0})={\bf 0}$.
Then, $f$ is totally geodesic by \cite[Proposition 1]{CM16} and $f(\mathbb B^n)$ is indeed an affine linear section of $\Omega$ in $\mathbb C^N$ (cf. \cite{Mok12}).
Therefore, the result follows when $r=2$.
Now, we suppose that $r=3$.
If $\lambda'=1$ or $\lambda'=3$, then the result follows from Proposition 1 and Theorem 1 in \cite{CM16}.
If $\lambda'=2$, then the result follows from Theorem \ref{LinearSection_Ball_IsoConsteq2}.
\end{proof}
\begin{remark}
In general, we expect that Theorem 1 in \cite{CM16} holds true for any holomorphic isometry from $(\mathbb B^n,k g_{\mathbb B^n})$ to $(\Omega,g_\Omega)$ for $1\le k\le \mathrm{rank}(\Omega)$.
Actually, Theorem \ref{LS_rk_2or3} asserts that our expectation is true when $\Omega$ is an irreducible bounded symmetric domain of rank at most three.
Moreover, Theorem \ref{LS_rk_2or3} also holds true for any holomorphic isometry from $(\Delta,kds_\Delta^2)$ to $(\Delta^p,ds_{\Delta^p}^2)$ for any positive integer $k$ and any integer $p$ such that $2\le p\le 3$.
However, for $2\le p\le 3$ one may make use of Ng's classification of all holomorphic isometries from $(\Delta,kds_\Delta^2)$ to $(\Delta^p,ds_{\Delta^p}^2)$ (cf. \cite{Ng10}) to prove such an analogue of Theorem \ref{LS_rk_2or3}.

On the other hand, when $\Omega\Subset \mathbb C^N$ is an irreducible bounded symmetric domain of rank $r\ge 4$, it is not known that whether all holomorphic isometries in $\widehat{\bf HI}_k(\mathbb B^n,\Omega)$ arise from linear sections of the minimal embedding of the compact dual $X_c$ of $\Omega$ for $3\le k\le r-1$. In other words, the problem remains open for the space $\widehat{\bf HI}_k(\mathbb B^n,\Omega)$ when $\Omega$ is of rank $r\ge 4$ and $3\le k\le r-1$.
\end{remark}

Now, we would like to emphasise the following consequence of both Theorem \ref{Thm_UB_PR} and Theorem \ref{LS_rk_2or3}.
\begin{corollary}\label{Cor:Consequence1}
Let $\Omega\Subset \mathbb C^N$ be an irreducible bounded symmetric domain of rank $\ge 2$ in its Harish-Chandra realization such that $n(\Omega)\le p(\Omega)$.
If $f\in \widehat{\bf HI}_{\lambda'}(\mathbb B^n,\Omega)$ for some real constant $\lambda'>0$, then we have the following:
\begin{enumerate}
\item $n\le p(\Omega)$ when $\lambda'\ge 2$; $n\le p(\Omega)+1$ when $\lambda'=1$,
\item $f(\mathbb B^n)$ is an irreducible component of some complex-analytic subvariety $\mathscr V\subset \Omega$ satisfying $\iota(\mathscr V)=P\cap \iota(\Omega)$, where $\iota:X_c\hookrightarrow \mathbb P\big(\Gamma(X_c,\mathcal O(1))^*\big)\cong\mathbb P^{N'}$ is the minimal embedding and $P\subseteq \mathbb P\big(\Gamma(X_c,\mathcal O(1))^*\big)\cong\mathbb P^{N'}$ is some projective linear subspace.
\end{enumerate}
\end{corollary}
\begin{proof}
Note that $(1)$ follows from Theorem \ref{Thm_UB_PR} when $\lambda'\ge 2$.
On the other hand, $(1)$ follows from Theorem $2$ in \cite{Mok16} when $\lambda'=1$.
Moreover, $(2)$ follows from Theorem \ref{LS_rk_2or3} because $\Omega$ is of rank at most three whenever $\Omega$ satisfies $n(\Omega)\le p(\Omega)$.
\end{proof}
\begin{remark}\label{remark3.14}
\begin{enumerate}
\item[(1)]
In particular, Corollary \ref{Cor:Consequence1} holds true when $\Omega$ is either of type $\mathrm{IV}$ or of the exceptional type by Lemma \ref{NullDim1}.
From the method used in this section, it is not known whether both part $(1)$ and part $(2)$ of Corollary \ref{Cor:Consequence1} still hold true in general when the assumption $n(\Omega)\le p(\Omega)$ is removed.
\item[(2)]
Recently, Yuan \cite{Yuan17} pointed out to the author that one may obtain upper bounds on dimensions of isometrically embedded complex unit balls into irreducible bounded symmetric domains $\Omega$ of rank $\ge 2$ by using the functional equation for any holomorphic isometry $f:(\mathbb B^n,k g_{\mathbb B^n})\to (\Omega,g_{\Omega})$, $k\ge 2$, with $f({\bf 0})={\bf 0}$ and the signature of the sum of squares (cf. \cite[Proposition 2.11]{XY16}).
When the target is $\Omega=D^{\mathrm{I}}_{3,4}$, it suffices to consider the case where $k=2$ and we compute $m_2(D^{\mathrm{I}}_{3,4}) = {3 \choose 2}{4\choose 2} = 18$ by \cite{FHX16}.
(Noting that $\Omega=D^{\mathrm{I}}_{3,4}$ does not satisfy $n(\Omega)\le p(\Omega)$.)
Moreover, one may make use of the signature of the sum of squares (cf. \cite[Proposition 2.11]{XY16}) to conclude that ${n(n+1)\over 2} \le m_2(D^{\mathrm{I}}_{3,4}) = {3 \choose 2}{4\choose 2} = 18$, i.e., $n\le 5 = p(D^{\mathrm{I}}_{3,4})$.
In other words, combining with the results of the present article, both part $(1)$ and part $(2)$ of Corollary \ref{Cor:Consequence1} holds true for $\Omega=D^{\mathrm{I}}_{3,4}$.
Moreover, in general this method does not imply that $n \le p(\Omega)$ if there exists a holomorphic isometry $f:(\mathbb B^n,k g_{\mathbb B^n})\to (\Omega,g_{\Omega})$ with $k\ge 2$, where $\Omega$ is any irreducible bounded symmetric domain of rank $\ge 2$.
\end{enumerate}
\end{remark}
\section{On holomorphic isometries of complex unit balls into certain irreducible bounded symmetric domains of rank $2$}
\subsection{Characterization of images of holomorphic isometries}\label{Sec:4.1}
We start with the following lemma which identifies those irreducible bounded symmetric domains $\Omega\Subset \mathbb C^N$ of rank $2$ which carry extra properties.
\begin{lemma}\label{SpecialRank2}
Let $\Omega\Subset\mathbb C^N$ be an irreducible bounded symmetric domain of rank $2$ in its Harish-Chandra realization.
Then, $2N > N'+1$ provided that $\Omega$ is not biholomorphic to $D^{\mathrm{I}}_{2,q}$ for any $q \ge 5$.
\end{lemma}
\begin{proof}
The proof follows from direct computation for any irreducible bounded symmetric domain $\Omega$ of rank $2$ by using results in \cite[p.\;663]{NT76}.
Actually, we obtain from \cite{NT76} the value of $N':=N(1)$ for any irreducible Hermitian symmetric space $X_c$ of the compact type.
\paragraph{Case 1:} When $\Omega$ is not biholomorphic to any type-$\mathrm{I}$ domains $D^{\mathrm{I}}_{2,q}$ for $q\ge 3$, $\Omega$ is either biholomorphic to $D^{\mathrm{IV}}_m$ (for some $m\ge 3$), $D^{\mathrm{II}}_5$ or $D^{\mathrm{V}}$ because of $D^{\mathrm{IV}}_4 \cong D^{\mathrm{I}}_{2,2}$, $D^{\mathrm{IV}}_6\cong D^{\mathrm{II}}_4$ and $D^{\mathrm{III}}_2\cong D^{\mathrm{IV}}_3$.
If $\Omega\cong D^{\mathrm{IV}}_m$, $m\ge 3$, then it is clear that
$2m > N'+1 = m+2$.
If $\Omega\cong D^{\mathrm{II}}_5$, then
$2 \dim_{\mathbb C} D^{\mathrm{II}}_5 = 20
 > N'+1 = 2^{5-1}=16$.
If $\Omega\cong D^{\mathrm{V}}$, then
$2 \dim_{\mathbb C} D^{\mathrm{V}} = 32
 > N'+1 
 = 26+1=27$,
where $X_c$ is the compact dual of $D^{\mathrm{V}}$.
Thus, any such $\Omega$ satisfies the desired property.
\paragraph{Case 2:}
When $\Omega\cong D^{\mathrm{I}}_{2,q}$ for some $q\ge 3$, we have
$4q=2N > N'+1
= {2+q\choose q}= {(q+1)(q+2)\over 2}$
if and only if
$0 > q^2-5q+2=(q-{5\over 2})^2-{17\over 4}$, which is equivalent to $q=3$ or $q=4$ because $q\ge 3$ is an integer and $(q-{5\over 2})^2 \ge {25\over 4} > {17\over 4}$ for $q\ge 5$.
The result follows.
\end{proof}
\begin{remark}
We consider rank-$2$ irreducible bounded symmetric domains $\Omega$ because the functional equations of holomorphic isometries from $(\mathbb B^n,g_{\mathbb B^n})$ to $(\Omega,g_\Omega)$ are similar to those of holomorphic isometries from $(\mathbb B^n,g_{\mathbb B^n})$ to $(D^{\mathrm{IV}}_m,g_{D^{\mathrm{IV}}_m})$ for $m\ge 3$ under the assumption that the isometries map ${\bf 0}$ to ${\bf 0}$. This is related to the study in \cite{CM16}.
In addition, we will assume that such a bounded symmetric domain $\Omega$ satisfies $2\cdot \dim_{\mathbb C}\Omega> N'+1$.
\end{remark}

Let $\Omega\Subset\mathbb C^N$ be an irreducible bounded symmetric domain of rank $2$ in its Harish-Chandra realization satisfying
$2N > N'+1$, where $N':=\dim_{\mathbb C} \mathbb P\big(\Gamma(X_c,\mathcal O(1))^*\big)$ and $X_c$ is the compact dual Hermitian symmetric space of $\Omega$.
Recall that $g_\Omega$ is the canonical K\"ahler-Einstein metric on $\Omega$ normalized so that minimal disks are of constant Gaussian curvature $-2$.
In terms of the Harish-Chandra coordinates $z=(z_1,\ldots,z_N)\in \Omega\subset \mathbb C^N$, the K\"ahler form with respect to $g_\Omega$ is equal to $\omega_{g_\Omega} = -\sqrt{-1}\partial\overline\partial \log h_\Omega(z,z)$, where
\[ h_\Omega(z,\xi)
= 1-\sum_{j=1}^N z_j\overline{\xi_j} + 
\sum_{l=1}^{N'-N} \hat G_l(z)\overline{\hat G_l(\xi)} \]
such that each $\hat G_l(z)$ is a homogeneous polynomial of degree $2$ in $z$, i.e., $\hat G_l(\lambda z) = \lambda^2 \hat G_l(z)$ for any $\lambda\in \mathbb C^*$.
Note that from Section \ref{Sec:2}, we have $G_{l+N}(z)=\hat G_{l}(z)$ for $l=1,\ldots,N'-N$.
Write ${\bf G}(z) := \begin{pmatrix}
\hat G_1(z),\ldots, \hat G_{N'-N}(z)
\end{pmatrix}^T$.
Let $n,N$ and $N'$ be positive integers satisfying $N'-N+n\le N$.
We also let ${\bf U'}\in M(N-n,N;\mathbb C)$ be such that $\mathrm{rank}({\bf U'})=N-n$.
Then, we define
\begin{equation}\label{EqWU'}
W_{\bf U'}
:=\left\{(z_1,\ldots,z_N)\in \Omega:
 {\bf U'} \begin{pmatrix}
z_1\\\vdots \\ z_N
\end{pmatrix} = \begin{pmatrix}
{\bf G}(z) \\{\bf 0}_{(2N-n-N') \times 1}
\end{pmatrix} \right\}.
\end{equation}
The following generalizes the study of $\widehat{\bf HI}_1(\mathbb B^n,D^{\mathrm{IV}}_N)$, $N\ge 3$, in Chan-Mok \cite{CM16}.
Moreover, in the following proposition,
the reason of assuming $n\le 2N-N' =:n_0(\Omega)$ is that
there is a certain explicitly defined class of complex-analytic subvarieties of $\Omega$ which contains the images of all holomorphic isometries $(\mathbb B^n,g_{\mathbb B^n})\to (\Omega,g_\Omega)$ up to composing with elements in $\mathrm{Aut}(\Omega)$, and each of them is contained entirely in $W_{\bf U''}$ for some matrix ${\bf U''}\in M(N-n_0(\Omega),N;\mathbb C)$ satisfying ${\bf U''}\overline{\bf U''}^T={\bf I_{N-n_0(\Omega)}}$.
We will show that this gives a relation between the spaces $\widehat{\bf HI}_1(\mathbb B^n,\Omega)$, $1\le n\le n_0(\Omega)-1$, and $\widehat{\bf HI}_1(\mathbb B^{n_0(\Omega)},\Omega)$.
\begin{proposition}\label{CharCUBintoIrrBSDrk2}
Let $\Omega\Subset \mathbb C^N$ be an irreducible bounded symmetric domain of rank $2$ in its Harish-Chandra realization such that $2N
> N'+1$, where $N':=\dim_{\mathbb C} \mathbb P\big(\Gamma(X_c,\mathcal O(1))^*\big)$ and $X_c$ is the compact dual Hermitian symmetric space of $\Omega$.
Let $n$ be an integer satisfying $1\le n$ $\le$ $2N$ $-$ $N'$.
If $f\in \widehat{\bf HI}_1(\mathbb B^n,\Omega)$, then $\Psi(f(\mathbb B^n))$ is the irreducible component of $W_{\bf U'}$ containing ${\bf 0}$ for some matrix
${\bf U'}\in M(N-n,N;\mathbb C)$ satisfying ${\bf U'}\overline{\bf U'}^T={\bf I_{N-n}}$ and some $\Psi\in \mathrm{Aut}(\Omega)$ satisfying $\Psi(f({\bf 0}))={\bf 0}$.
Conversely, given any matrix ${\bf U''}\in M(N-n,N;\mathbb C)$ satisfying ${\bf U''}\overline{\bf U''}^T={\bf I_{N-n}}$, the irreducible component of $W_{\bf U''}$ containing ${\bf 0}$ is the image of some $\widetilde f\in \widehat{\bf HI}_1(\mathbb B^n,\Omega)$.
\end{proposition}
\begin{proof}
Let $f\in \widehat{\bf HI}_1(\mathbb B^n,\Omega)$.
Assume without loss of generality that $f({\bf 0})={\bf 0}$.
Then, we have
\[ 1-\sum_{j=1}^N |f^j(w)|^2 + 
\sum_{l=1}^{N'-N} |\hat G_l(f(w))|^2 = 1-\sum_{l=1}^n |w_l|^2. \]
Note that $2N-1\ge N'$ and $2N-N'\ge n$.
By Lemma \ref{LemCaSOS}, there exists ${\bf U}\in U(N)$ such that
\begin{equation}\label{Eq7211}
{\bf U} \begin{pmatrix}
f^1(w),\ldots,f^N(w)
\end{pmatrix}^T = \begin{pmatrix}
w_1,\ldots, w_n, {\bf G}(f(w))^T, {\bf 0}_{1\times (2N-n-N')}
\end{pmatrix}^T.
\end{equation}
We write ${\bf U}=\begin{bmatrix}
{\bf A'}\\{\bf U'}
\end{bmatrix}$ so that ${\bf U'}\in M(N-n,N;\mathbb C)$ is a matrix which satisfies ${\bf U'}\overline{\bf U'}^T={\bf I_{N-n}}$.
Then, we have $f(\mathbb B^n)\subseteq W_{\bf U'}$ by $\text{Eq.}\; (\ref{Eq7211})$.
It is clear that the Jacobian matrix of $W_{\bf U'}$ at ${\bf 0}$ is equal to ${\bf U'}$, which is of full rank $N-n$ so that $W_{\bf U'}$ is smooth at ${\bf 0}$ and of dimension $n$ at ${\bf 0}$.
Let $S$ be the irreducible component of $W_{\bf U'}$ containing $f(\mathbb B^n)$, which also contains ${\bf 0}$.
Then, we have $\dim S=n$.
Since both $S$ and $f(\mathbb B^n)$ are irreducible complex-analytic subvarieties of $\Omega$, $f(\mathbb B^n)\subseteq S$ and $\dim S= \dim f(\mathbb B^n) = n$, we have
$S=f(\mathbb B^n)$.
Thus, the irreducible component of $W_{\bf U'}$ containing ${\bf 0}$ is the image of some holomorphic isometric embedding $f:(\mathbb B^n,g_{\mathbb B^n})\to (\Omega,g_\Omega)$.

Conversely, let $n$ be an integer satisfying $1\le n\le 2N-N'$ and let ${\bf U''}\in M(N-n,N;\mathbb C)$ be a matrix satisfying ${\bf U''}\overline{\bf U''}^T={\bf I_{N-n}}$.
By Lemma \ref{Matrix2}, there exists ${\bf A''}\in M(n,N;\mathbb C)$ such that $\begin{bmatrix}
{\bf A''}\\{\bf U''}
\end{bmatrix}\in U(N)$ so that
\begin{equation}\label{Eq7212}
 \begin{bmatrix}
{\bf A''}\\{\bf U''}
\end{bmatrix}
\begin{pmatrix}
z_1\\\vdots \\ z_N
\end{pmatrix}
=\begin{pmatrix} 
{\bf w}(z)\\
{\bf G}(z)\\
{\bf 0}_{(2N-n-N') \times 1}
\end{pmatrix}\quad \forall\;z=(z_1,\ldots,z_n) \in W_{\bf U''} ,
\end{equation}
where
${\bf w}(z)=\begin{pmatrix}
w_1(z),\ldots,w_n(z)
\end{pmatrix}^T :={\bf A''}\begin{pmatrix}
z_1,\ldots,z_N
\end{pmatrix}^T$.
Note that the Jacobian matrix of $W_{\bf U''}$ at ${\bf 0}$ is equal to ${\bf U''}$, which is of full rank $N-n$ so that $W_{\bf U''}$ is smooth at ${\bf 0}$ and of dimension $n$ at ${\bf 0}$.
Let $S'$ be the irreducible component of $W_{\bf U''}$ containing ${\bf 0}$.
Then, we have $\dim S'=n$.
Actually $S'$ is precisely the point set closure of the connected component of $\mathrm{Reg}(W_{\bf U''})$ containing ${\bf 0}$ in $\Omega$.
Denote by $\mathrm{Reg}(S')$ the regular locus of $S'$.
Then, $\mathrm{Reg}(S')$ is a connected complex manifold lying inside $\Omega$ and ${\bf 0}\in \mathrm{Reg}(S')$.
Let $\varphi:B({\bf 0})\to \mathrm{Reg}(S')$ be a biholomorphism onto an open neighborhood of ${\bf 0}$ in $\mathrm{Reg}(S')$ such that $\varphi({\bf 0})={\bf 0}$, where $B({\bf 0})$ is some open neighborhood of ${\bf 0}$ in $\mathbb C^n$.
Here the image $\varphi(B({\bf 0}))$ is a germ of complex submanifold of $\Omega$ at ${\bf 0}$, i.e., a complex submanifold of some open neighborhood of ${\bf 0}$ in $\Omega$.
Note that
$h_\Omega(z,z) = 1-\sum_{l=1}^n |w_l(z)|^2$
for any $z\in S'$ and $\zeta=(\zeta_1,\ldots,\zeta_n)$ can be regarded as local holomorphic coordinates on $\mathrm{Reg}(S')$ around ${\bf 0}\in \mathrm{Reg}(S')$.
Then, it follows from Eq. (\ref{Eq7212}) that for $\zeta\in B({\bf 0})$, we have
\begin{equation}\label{Eq4.4}
h_\Omega(\varphi(\zeta),\varphi(\zeta)) = 1-\sum_{l=1}^n |w_l(\varphi(\zeta))|^2
\end{equation}
and $-\log h_\Omega(\varphi(\zeta),\varphi(\zeta))
= -\log\left(1-\sum_{l=1}^n |w_l(\varphi(\zeta))|^2 \right)$ is a local K\"ahler potential on $\mathrm{Reg}(S')$ which is the restriction of the K\"ahler potential on $(\Omega,g_\Omega)$ to an open neighborhood of ${\bf 0}$ in $\mathrm{Reg}(S')$.
It follows from Eq. (\ref{Eq4.4}) that the germ of $S'$ at ${\bf 0}$ is the image of a germ of holomorphic isometry $\widetilde f:(\mathbb B^n,g_{\mathbb B^n};{\bf 0})\to (\Omega,g_\Omega;{\bf 0})$.
By the extension theorem of Mok \cite{Mok12}, $\widetilde f$ extends to a holomorphic isometric embedding $\widetilde f:(\mathbb B^n,g_{\mathbb B^n})\to (\Omega,g_\Omega)$.
Since both $\widetilde f(\mathbb B^n)$ and $S'$ are $n$-dimensional irreducible complex-analytic subvarieties of $\Omega$ and
$\widetilde f(B^n({\bf 0},\varepsilon))\subset \widetilde f(\mathbb B^n)\cap S'$ for some real number $\varepsilon\in (0,1)$.
It follows that $S'=\widetilde f(\mathbb B^n)$.
Hence, the irreducible component of $W_{\bf U''}$ containing ${\bf 0}$ is the image of some holomorphic isometric embedding $\widetilde f \in \widehat{\bf HI}_1(\mathbb B^n,\Omega)$.
\end{proof}
\begin{remark}
From the proof of Lemma \ref{SpecialRank2}, we see that Proposition \ref{CharCUBintoIrrBSDrk2} precisely holds true for the space $\widehat{\bf HI}_1(\mathbb B^n,\Omega)$ whenever the integer $n$ and the bounded symmetric domain $\Omega$ satisfy one of the following:
\begin{enumerate}
\item $\Omega \cong D^{\mathrm{I}}_{2,3}$, $1\le n\le 3=p(D^{\mathrm{I}}_{2,3})$,
\item $\Omega\cong D^{\mathrm{I}}_{2,4}$, $1\le n\le 2$,
\item $\Omega\cong D^{\mathrm{II}}_5$, $1\le n \le 5 = p(D^{\mathrm{II}}_5)-1$,
\item $\Omega \cong D^{\mathrm{IV}}_m$ for some integer $m\ge 3$, $1\le n\le m-1=p(D^{\mathrm{IV}}_m)+1$,
\item $\Omega\cong D^{\mathrm{V}}$, $1\le n\le 6$.
\end{enumerate}
\noindent Moreover, Proposition \ref{CharCUBintoIrrBSDrk2} actually provides the classification of images of all $f\in \widehat{\bf HI}_1(\Delta,\Omega)$ whenever $\Omega$ is a rank-$2$ irreducible bounded symmetric domain which is not biholomorphic to $D^{\mathrm{I}}_{2,q}$ for any $q\ge 5$.
This also solves part of Problem $3$ in \cite[p. 2645]{MN10} theoretically.
It is excepted that there are many incongruent holomorphic isometries in $\widehat{\bf HI}_1(\Delta,\Omega)$.
However, Proposition \ref{CharCUBintoIrrBSDrk2} at least provides a source of constructing explicit examples of holomorphic isometries in $\widehat{\bf HI}_1(\Delta,\Omega)$.
In particular, for the case where the target is an irreducible bounded symmetric domain of rank $2$, Problem $3$ in \cite[p. 2645]{MN10} remains unsolved precisely in the case where the target $\Omega$ is $D^{\mathrm{I}}_{2,q}$ for some $q\ge 5$.
\end{remark}
\subsection{Proof of Theorem \ref{Thm_Slicing_for_special_cases1}}\label{Sec:4.2}
As we have mentioned in Section \ref{Sec:4.1}, Proposition \ref{CharCUBintoIrrBSDrk2} actually gives a relation between the spaces $\widehat{\bf HI}_1(\mathbb B^n,\Omega)$, $1\le n\le n_0(\Omega)-1$, and $\widehat{\bf HI}_1(\mathbb B^{n_0(\Omega)},\Omega)$. In other words, this yields Theorem \ref{Thm_Slicing_for_special_cases1}.
\begin{proof}[Proof of Theorem \ref{Thm_Slicing_for_special_cases1}]
We follow the setting in the proof of Proposition \ref{CharCUBintoIrrBSDrk2}.
Assume without loss of generality that $f({\bf 0})={\bf 0}$.
Note that $N'-N+n<N$ and thus $f(\mathbb B^n)$ is the irreducible component of $W_{\bf U'}$ containing ${\bf 0}$ for some matrix ${\bf U'}\in M(N-n,N;\mathbb C)$ satisfying ${\bf U'}\overline{\bf U'}^T={\bf I_{N-n}}$ by Proposition \ref{CharCUBintoIrrBSDrk2}.
Moreover, we have
\[ \begin{bmatrix}
{\bf A'}\\{\bf U'}
\end{bmatrix}\begin{pmatrix}
f^1(w),\ldots,f^N(w)
\end{pmatrix}^T = \begin{pmatrix}
w_1,\ldots,w_n, {\bf G}(f(w))^T, {\bf 0}_{1\times (2N-N'-n)}
\end{pmatrix}^T \]
for some ${\bf A'}\in M(n,N;\mathbb C)$ such that $\begin{bmatrix}
{\bf A'}\\{\bf U'}
\end{bmatrix}\in U(N)$ after composing with some element in the isotropy subgroup of $\mathrm{Aut}(\mathbb B^n)$ at ${\bf 0}$ if necessary (by Lemma \ref{Matrix2}).
We write ${\bf U'}=\begin{bmatrix}
{\bf U'_1} \\ {\bf U'_2}
\end{bmatrix}$ for some ${\bf U'_1} \in M(N'-N,N;\mathbb C)$ and ${\bf U'_2}\in M(2N-N'-n,N;\mathbb C)$.
Moreover, we have
${\bf U'_1}\begin{pmatrix}
z_1,\ldots, z_N
\end{pmatrix}^T
= {\bf G}(z)$
and ${\bf U'_1}\overline{\bf U'_1}^T = {\bf I_{N'-N}}$ for any $z\in W_{\bf U'}$.
It follows from Proposition \ref{CharCUBintoIrrBSDrk2} that the irreducible component of $W_{\bf U'_1}$ containing ${\bf 0}$ is the image of some holomorphic isometric embedding $F:(\mathbb B^{n_0},g_{\mathbb B^{n_0}})\to (\Omega,g_\Omega)$, where $n_0=n_0(\Omega):=2N-N'$.
We may suppose that $F({\bf 0})={\bf 0}$ without loss of generality.
Since $f(\mathbb B^n)\subset \Omega$ is irreducible and $f(\mathbb B^n)\subset W_{\bf U'_1}$, $S:=f(\mathbb B^n)$ lies inside the irreducible component $S':=F(\mathbb B^{n_0})$ of $W_{\bf U'_1}$ containing ${\bf 0}$.
Since $(S,g_\Omega|_S)\cong (\mathbb B^n,g_{\mathbb B^n})$ and $(S',g_\Omega|_{S'})\cong (\mathbb B^{n_0},g_{\mathbb B^{n_0}})$ are of constant holomorphic sectional curvature $-2$, $(S,g_\Omega|_S)\subset (S',g_\Omega|_{S'})$ is totally geodesic and the result follows (cf. the proof of \cite[Theorem 2]{CM16}).
\end{proof}
\begin{remark}\text{}
\begin{enumerate}
\item[(1)]
It follows from Lemma \ref{SpecialRank2} that Theorem \ref{Thm_Slicing_for_special_cases1} holds true when the pair $(\Omega,n_0(\Omega))$ is one of the following:
\begin{enumerate}
\item $\Omega \cong D^{\mathrm{I}}_{2,3}$, $n_0(\Omega)=3$,
\item $\Omega\cong D^{\mathrm{I}}_{2,4}$, $n_0(\Omega)=2$,
\item $\Omega\cong D^{\mathrm{II}}_5$, $n_0(\Omega)=5$,
\item $\Omega \cong D^{\mathrm{IV}}_m$ ($m\ge 3$), $n_0(\Omega)=m-1$,
\item $\Omega\cong D^{\mathrm{V}}$, $n_0(\Omega)=6$.
\end{enumerate}
\item[(2)] It is not known whether Theorem \ref{Thm_Slicing_for_special_cases1} still holds true when $n_0(\Omega)$ was replaced by $p(\Omega)+1$ and $\Omega\not\cong D^{\mathrm{IV}}_m$ for any integer $m\ge 3$.
\item[(3)]
For the particular case where $\Omega=D^{\mathrm{I}}_{2,3}$, it follows from \cite{Mok16} that if the space $\widehat{\bf HI}_1(\mathbb B^n,D^{\mathrm{I}}_{2,3})$ is non-empty, then $n\le p(D^{\mathrm{I}}_{2,3})+1=4$.
In this case, it is motivated by our study in the present article to consider the following problem in order to classify all holomorphic isometries in $\widehat{\bf HI}_1(\mathbb B^n,D^{\mathrm{I}}_{2,3})$:

\noindent Given any $f\in \widehat{\bf HI}_1(\mathbb B^3,D^{\mathrm{I}}_{2,3})$, could $f$ be factorized as $f=F\circ \rho$ for some
$F\in \widehat{\bf HI}_1(\mathbb B^4,D^{\mathrm{I}}_{2,3})$ and $\rho\in \widehat{\bf HI}_1(\mathbb B^3,\mathbb B^4)$?

\noindent If the problem were solved and the answer were affirmative, then the classification of all holomorphic isometries in $\widehat{\bf HI}_1(\mathbb B^n,D^{\mathrm{I}}_{2,3})$ would be reduced to the uniqueness problem for non-standard (i.e., not totally geodesic) holomorphic isometries in $\widehat{\bf HI}_1(\mathbb B^4,D^{\mathrm{I}}_{2,3})$ constructed by Mok \cite{Mok16}.
\end{enumerate}
\end{remark}

\begin{center}
\textsc{Acknowledgement}
\end{center}
This work is part of the author's Ph.D. thesis \cite{Ch16} at The University of Hong Kong except for item $(2)$ of Remark \ref{remark3.14}.
He would like to express his gratitude to his supervisor, Professor Ngaiming Mok, for his guidance and encouragement.
The author would also like to thank Dr. Yuan Yuan for his interest in the research which leads to item $(2)$ of Remark \ref{remark3.14}.


\begin{thebibliography}{xxxxxxx}
\bibitem[Ca53]{Ca53} E. Calabi:
\emph{Isometric imbedding of complex manifolds},
Ann. of Math. {\bf 58} (1953), 1-23.
\bibitem[Ch16]{Ch16} S.-T. Chan: \emph{On holomorphic isometric embeddings of complex unit balls into bounded symmetric domains}, Ph.D. thesis at The University of Hong Kong, 2016.
\bibitem[CM16]{CM16} S.-T. Chan and N. Mok: \emph{Holomorphic isometries of $\mathbb B^m$ into bounded symmetric domains arising from linear sections of minimal embeddings of their compact duals}, Math. Z., DOI 10.1007/s00209-016-1778-7.
\bibitem[FHX16]{FHX16} H. Fang, X. Huang and M. Xiao:
\emph{Volume-preserving maps between Hermitian
symmetric spaces of compact type},
arXiv:1602.01900.
\bibitem[HM99]{HM99} J.-M. Hwang and N. Mok: \emph{Varieties of minimal rational tangents on uniruled projective manifolds}, Several complex variables (Berkeley, CA, 1995-1996), 351-389, Math. Sci. Res. Inst. Publ., 37, Cambridge Univ. Press, Cambridge, 1999.
\bibitem[KM08]{KM08} V. Koziarz and J. Maubon: \emph{Representations of complex hyperbolic lattices into rank $2$ classical Lie groups of Hermitian type}, Geom. Dedicata 137 (2008), 85-111.
\bibitem[Lo77]{Lo77} O. Loos:
\emph{Bounded symmetric domains and Jordan pairs}, Math. Lectures, Univ. of California, Irvine, 1977.
\bibitem[Mk89]{Mok89} N. Mok: \emph{Metric rigidity theorems on Hermitian locally symmetric manifolds},
Series in Pure Mathematics, Vol. 6, World Scientific Publishing Co., Singapore; Teaneck, NJ, 1989.
\bibitem[Mk02a]{Mok2002Compositio} N. Mok: \emph{Characterization of certain holomorphic geodesic cycles on quotients of bounded symmetric domains in terms of tangent subspaces},
Compositio Math. 132 (2002), 289-309.
\bibitem[Mk02b]{Mok2002} N. Mok:
\emph{Local holomorphic isometric embeddings arising from correspondences in the rank-$1$ case},
contemporary trends in algebraic geometry and algebraic topology (Tianjin, 2000), 155-165, Nankai Tracts Math., 5, World Sci. Publ., River Edge, NJ,	Singapore, 2002.
\bibitem[Mk11]{Mok11} N. Mok: \emph{Geometry of holomorphic isometries and related maps between bounded domains},
Geometry and analysis. No. 2, 225-270, Adv. Lect. Math. (ALM), 18, Int. Press, Somerville, MA, 2011.
\bibitem[Mk12]{Mok12} N. Mok: \emph{Extension of germs of holomorphic isometries up to normalizing constants with respect to the Bergman metric}, 
J. Eur. Math. Soc. (JEMS) 14 (2012), 1617-1656.
\bibitem[Mk14]{Mok14} N. Mok: \emph{Local holomorphic curves on a bounded symmetric domain in its Harish-Chandra realization exiting at regular points of the boundary},
Pure Appl. Math. Q. 10 (2014), 259-288. 
\bibitem[Mk16]{Mok16} N. Mok: \emph{Holomorphic isometries of the complex unit ball into irreducible bounded symmetric domains}, Proc. Amer. Math. Soc. 144 (2016), 4515-4525.
\bibitem[MN10]{MN10} N. Mok and S.-C. Ng:
\emph{Second fundamental forms of holomorphic isometries of the Poincar\'e disk into bounded symmetric domains and their boundary behavior along the unit circle}, 
Sci. China Ser. A 52 (2009), 2628-2646.
\bibitem[Mu95]{Mum95} D. Mumford:
\emph{Algebraic geometry I: complex projective varieties}, Springer-Verlag, Berlin, 1995.
\bibitem[NT76]{NT76} H. Nakagawa and R. Takagi: \emph{On locally symmetric K\"ahler submanifolds in a complex projective space},
J. Math. Soc. Japan 28 (1976), 638-667.
\bibitem[Ng10]{Ng10} S.-C. Ng:
\textit{On holomorphic isometric embeddings of the unit disk into polydisks},
Proc. Amer. Math. Soc. 138 (2010), 2907-2922.
\bibitem[Ng11]{Ng11} S.-C. Ng:
\emph{On holomorphic isometric embeddings of the unit $n$-ball into products of two unit $n$-balls},
Math. Z. 268 (2011), 347-354.
\bibitem[UWZ16]{UWZ16} H. Upmeier, K. Wang and G. Zhang: \emph{Holomorphic isometries from the unit ball into symmetric domains}, arXiv:1603.03639.
\bibitem[Wo72]{Wo72} J. A. Wolf: \emph{Fine structures of Hermitian symmetric spaces},
Symmetric spaces (Short Courses, Washington Univ., St. Louis, Mo., 1969-1970), pp. 271-357. Pure and App. Math., Vol. 8, Dekker, New York, 1972.
\bibitem[XY16]{XY16} M. Xiao and Y. Yuan: \emph{Holomorphic maps from the complex unit ball to Type $\mathrm{IV}$ classical domains}, arXiv:1606.04806.
\bibitem[Y17]{Yuan17} Y. Yuan: Private communications.
\bibitem[Zh99]{FuzhenZhang} F. Zhang:
\emph{Matrix Theory, basic results and techniques},
Springer-Verlag, New York, 1999.
\end{thebibliography}
\end{document}